\documentclass[a4paper,12pt]{amsart}
\usepackage{geometry}

\usepackage{amsmath,amsthm,amsfonts,amssymb,amscd}
\usepackage{fullpage}
\usepackage{lastpage}
\usepackage{enumerate}
\usepackage{fancyhdr}

\usepackage[american]{babel}

\usepackage{mathrsfs}
\usepackage{latexsym,epsfig,hyperref}          

\usepackage{pstricks}

\usepackage{enumitem}
\usepackage{palatino}

\usepackage{tikz-cd}
\usepackage{lipsum}

\newtheorem{thm}{Theorem}[section]
\newtheorem{remark}[thm]{Remark}
\newtheorem{prop}[thm]{Proposition}

\newtheorem{coro}[thm]{Corollary}
\newtheorem{claim}[thm]{Claim}
\newtheorem{lemma}[thm]{Lemma}

\newtheorem*{thmA}{Theorem A}
\newtheorem*{thmB}{Theorem B}

\newtheorem*{thmAm}{Theorem A$_m$}
\newtheorem*{thmBm}{Theorem B$_m$}

\def\dim{\operatorname{dim}}

\def\corank{\operatorname{corank}}

\def\Pic{\operatorname{Pic}}%
\def\Aut{\operatorname{Aut}}%
\def\Div{\operatorname{Div}}%
\def\Ext{\operatorname{Ext}}%

\title{On the Tschirnhausen module of coverings of curves on decomposable ruled surfaces and applications}

\author[Y. Choi]{Youngook Choi}
\address{Department of Mathematics Education, Yeungnam University, 280 Daehak-Ro, \hfill \newline\texttt{}
 \indent Gyeongsan, Gyeongbuk 38541, Republic of Korea}
\email{ychoi824@yu.ac.kr}

\author[H. Iliev]{Hristo Iliev}
\address{American University in Bulgaria, 2700 Blagoevgrad, Bulgaria, and \hfill \newline\texttt{}
 \indent Institute of Mathematics and Informatics, Bulgarian Academy of Sciences, \hfill \newline\texttt{}
 \indent 1113 Sofia, Bulgaria}
\email{ hiliev@aubg.edu, hki@math.bas.bg}

\author[S. Kim]{Seonja Kim}
\address{Department of Electronic Engineering, Chungwoon University, Sukgol-ro, Nam-gu, \hfill \newline\texttt{}
 \indent Incheon 22100, Republic of Korea}
\email{sjkim@chungwoon.ac.kr}

\thanks{The first author was supported by the National Research Foundation of Korea(NRF) grant funded by the Korea government(MSIT) (RS-2024-00352592).
The second author was supported by Grant KP-06-N 62/5 of Bulgarian National Science Fund.
The third  author was supported by the National Research Foundation of Korea(NRF) grant funded by the Korea government(MSIT) (2022R1A2C1005977).}

\date{\today}

\makeatletter
\@namedef{subjclassname@2020}{\textup{2020} Mathematics Subject Classification}
\makeatother

\usepackage{setspace}
\begin{document}

\setlength{\parindent}{5ex}

\begin{abstract}
We show that for two classes of $m$-secant curves $X \subset S$, with $m \geq 2$, where $f : S = \mathbb{P} (\mathcal{O}_Y \oplus \mathcal{O}_Y (E)) \to Y$ and $E$ is a non-special divisor on a smooth curve $Y$, the Tschirnhausen module $\mathcal{E}^{\vee}$ of the covering $\varphi = f_{|_X} : X \to Y$ decomposes completely as a direct sum of line bundles. Specifically, we prove that:
for $X \in |\mathcal{O}_S (mH)|$, where $H$ denotes the tautological divisor on $S$, one has
$
\mathcal{E}^{\vee} \cong \mathcal{O}_Y (-E) \oplus \cdots \oplus \mathcal{O}_Y (-(m-1)E)
$;
for $X \in |\mathcal{O}_S (mH + f^{\ast}q))|$, where $q$ is a point on $Y$,
$
 \mathcal{E}^{\vee} \cong \mathcal{O}_Y (-E-q) \oplus \cdots \oplus \mathcal{O}_Y (-(m-1)E-q)
$
holds.
This decomposition enables us to compute the dimension of the space of global sections of the normal bundle of the embedding $X \subset \mathbb{P}^R$ induced by the tautological line bundle $|\mathcal{O}_S (H)|$, where $R = \dim |\mathcal{O}_S (H)|$. As an application, we construct new families of generically smooth components of the Hilbert scheme of curves, including components whose general points correspond to non-linearly normal curves, as well as nonreduced components.
\end{abstract}


\maketitle

\section{Introduction}\label{Sec1}

Let $Y$  be a smooth irreducible projective curve of genus $\gamma$ and $E$ be a non-special divisor of degree $e$ on $Y$. Consider the ruled surface $S := \mathbb{P} (\mathcal{O}_Y \oplus \mathcal{O}_Y (E))$, following Hartshorne's conventions as in \cite{Hart77}, with natural projection morphism
\begin{equation}\label{f_S_to_Y}
 f : S = \mathbb{P} (\mathcal{O}_Y \oplus \mathcal{O}_Y (E)) \to Y .
\end{equation}
Let $\sigma_0 : Y \to S$ be the section corresponding to the surjection
$
 \mathcal{O}_Y \oplus \mathcal{O}_Y (E) \twoheadrightarrow \mathcal{O}_Y
$,
and let $Y_0 := \sigma_0 (Y)$ be the curve of minimum self-intersection on $S$. Note that $Y^2_0 = -e$. Denote by $H$ the tautological divisor on $S$, that is, the one for which $\mathcal{O}_{S} (1) \cong \mathcal{O}_{S} (H)$. Recall that
\[
 \Pic (S) \cong \mathbb{Z} [H] \oplus f^{\ast} (\Pic (Y)) ,
\]
and for a decomposable $S$ as in \eqref{f_S_to_Y}, we have
\[
 H \sim Y_0 + f^{\ast} E ,
\]
where $f^{\ast} D$ and $f^{\ast} L$ denote the corresponding pullback divisor and pullback line bundle for $D \in \Div (Y)$ and $L \in \Pic (Y)$.

Suppose that $X \subset S$ is an $m$-secant smooth curve on $S$, that is, $X \sim mY_0 + f^{\ast} D$, where $m \geq 1$ is an integer. In such a case,  the restriction of $f$ on $X$ determines an $m:1$ covering
\begin{equation*}\label{varphi_X_to_Y}
 \varphi := f_{|_X} : X \to Y .
\end{equation*}

Our first result concerns the case when $X \sim mH$.

\begin{thmA}\label{thmARef}
Let $Y$ be a smooth projective curve, and $E$ be a non-special divisor on $Y$. Consider the ruled surface $f : S = \mathbb{P} (\mathcal{O}_Y \oplus \mathcal{O}_Y (E)) \to Y$, and let $X \subset S$ be a smooth curve linearly equivalent to $mH$, where $H$ denotes the tautological divisor on $S$ and $m \geq 2$ is an integer. Denote by $\varphi := f_{|_X}$ the restriction of $f$ to $X$. Then $\varphi$ is a finite morphism of degree $m$, and we have the following decomposition:
 \begin{equation}\label{ThmA_Decomp}
  \varphi_{\ast} \mathcal{O}_X \cong \mathcal{O}_Y \oplus \mathcal{O}_Y (-E)  \oplus \cdots \oplus \mathcal{O}_Y (-(m-1)E) .
 \end{equation}
In particular, the Tschirnhausen module $\mathcal{E}^{\vee}$ associated to the covering $\varphi$ is given by:
\[
 \mathcal{E}^{\vee} \cong \mathcal{O}_Y (-E)  \oplus \cdots \oplus \mathcal{O}_Y (-(m-1)E) .
\]
\end{thmA}

\medskip

The proof of {\rm Theorem A} will be presented in Section~{\ref{Sec2}}. We note that although the decomposition of $\varphi_{\ast} \mathcal{O}_X$ in {\rm (\ref{ThmA_Decomp})} resembles that of a cyclic covering (see, for example, \cite[Remark 4.1.7, p.244]{Laz04}), the covering $\varphi : X \to Y$ is not, in general, cyclic when $m \geq 3$. We will return to this point in Section~{\ref{Sec4}}, following the proof of {\rm Theorem A$_m$}.

Our second result involving the Tschirnhausen module is about $X \sim mH + f^{\ast} q$, where $q \in Y$ is a point.

\begin{thmB}\label{thmBRef}
Let $Y$ be a smooth projective curve, and let $E$ be a non-special divisor on $Y$. Consider the ruled surface $f : S = \mathbb{P} (\mathcal{O}_Y \oplus \mathcal{O}_Y (E)) \to Y$, and let $H$ denote the tautological divisor on $S$. Let $m \geq 2$ be an integer, and let $q \in Y$ be a point. Suppose that $X \subset S$ is a smooth curve, linearly equivalent to $mH + f^{\ast} q$. Denote by $\varphi := f_{|_X}$ the restriction $f$ to $X$. Then:
\begin{enumerate}[label=(\roman*), leftmargin=*, font=\rmfamily]
 \item The morphism $\varphi$ is a finite cover of degree $m$, and its direct image satisfies
 \begin{equation*}\label{ThmB_Decomp}
  \varphi_{\ast} \mathcal{O}_X \cong \mathcal{O}_Y \oplus \mathcal{O}_Y (-E-q)  \oplus \cdots \oplus \mathcal{O}_Y (-(m-1)E-q) .
 \end{equation*}
In particular, the Tschirnhausen module $\mathcal{E}^{\vee}$ associated to the covering $\varphi$ is given by:
\[
 \mathcal{E}^{\vee} \cong \mathcal{O}_Y (-E-q)  \oplus \cdots \oplus \mathcal{O}_Y (-(m-1)E-q) .
\]

 \item The curve $X$ passes through the point $q_0 := f^{\ast} q \cap Y_0 \subset S$, where $Y_0$ is the section of minimal self-intersection. Moreover, for the sheaf $\mathcal{O}_X (q_0)$, its direct image under $\varphi$ is
 \begin{equation*}\label{ThmB_Decomp_q}
 \varphi_{\ast} \mathcal{O}_X (q_0) \cong  \mathcal{O}_Y \oplus \mathcal{O}_Y (-E) \oplus \mathcal{O}_Y (-2E-q)  \oplus \cdots \oplus \mathcal{O}_Y (-(m-1)E-q) \, .
 \end{equation*}
\end{enumerate}
\end{thmB}

\medskip

The proof of \textnormal{Theorem~B} will be given  in Section~\ref{Sec2}, following the proof of \textnormal{Theorem~A}.

Our interest in identifying the structure of the Tschirnhausen module stems from our previous works on the Hilbert scheme of curves - see \cite{CIK21}, \cite{CIK24a}, and \cite{CIK24b} - where we employed decompositions like those in \textnormal{Theorem~A} and \textnormal{Theorem~B} for the cases $m = 2$ and $m = 3$. In the case $m = 2$, the decomposition follows immediately from the fact that  $\mathcal{E}^{\vee}$ can be identified with the kernel of the trace map $\varphi_{\ast} \mathcal{O}_X \to \mathcal{O}_Y$; see, for example, \cite{Mir85} or \cite{CE96}. The decomposition we established for $m = 3$ in \cite{CIK24b}, further suggested that a similar structure might exist more generally for coverings of curves on decomposable ruled surfaces. In turn, \textnormal{Theorem~A}  and \textnormal{Theorem~B} enable us to generalize the results we previously obtained in \cite{CIK21} and \cite{CIK24b} concerning components of the Hilbert scheme of curves. To state these results, let $\mathcal{I}_{d,g,r}$ denote the union of irreducible components of the Hilbert scheme whose general points correspond to smooth, irreducible, non-degenerate curves of degree $d$ and genus $g$ in $\mathbb{P}^r$.

The following statement generalizes \cite[Theorem~A]{CIK21} to the case of $m:1$ coverings $\varphi : X \to Y$ of curves on cones, where $X$ is the intersection of the cone with a hypersurface of degree $m \geq 2$.

\begin{thmAm}
Assume that $m \geq 2$, $\gamma \geq 10$, and $e \geq 2\gamma + 1$ are integers. Define
\[
  d:= me , \quad g := \binom{m}{2}e + m\gamma + 1-m , \quad \mbox{ and } \quad R:= e - \gamma + 1 \, .
\]
\begin{enumerate}[label=(\roman*), leftmargin=*, font=\rmfamily]
 \item The Hilbert scheme $\mathcal{I}_{d, g, R}$ contains a \emph{generically smooth} irreducible component $\mathcal{D}_{m, R}$ of dimension
\[
  \dim \mathcal{D}_{m, R} = g + (m-1)(e - 2(\gamma-1)) - 1 + 3(\gamma-1) + R^2 + 2R \, .
\]
 A general point $[X_R] \in \mathcal{D}_{m, R}$ corresponds to a linearly normal curve that arises as the intersection of a general hypersurface of degree $m$ in $\mathbb{P}^{R}$ with a cone over a curve $Y \subset \mathbb{P}^{R-1}$ of degree $e$ and genus $\gamma$.

 \item If $r$ is an integer such that
\[
 \max \left\lbrace \gamma, \frac{2(e+2(\gamma-1))}{\gamma} \right\rbrace \leq r < R ,
\]
then the Hilbert scheme $\mathcal{I}_{d, g, r}$ contains a \emph{generically smooth} component $\mathcal{D}_{m, r}$ for which
\[
  \dim \mathcal{D}_{m, r} = g + (m-1)(e - 2(\gamma-1)) - 1 + 3(\gamma-1) + R r + R + r \, .
\]
 Moreover, a general point $[X_r] \in \mathcal{D}_{m, r}$ arises as a general projection of a curve $X_R$ as described in~\textnormal{(i)}.
\end{enumerate}
\end{thmAm}

\medskip

The structure of the proof of \textnormal{Theorem~A$_m$}  follows the approach of \cite[Theorem~A]{CIK21} and is presented in Section~\ref{Sec4}. We note that the case of a complete embedding, that is, $r = R$, was obtained by Flamini and Suppino in \cite{FS23}, as it was derived by them from \cite[Theorem~A]{CIK21} via induction on $m$, using degeneration of an $(m+1)$-secant smooth curve on the cone to an $m$-secant smooth curve and a hyperplane section of the cone. However, addressing the case of incomplete embedding for $m \geq 3$ requires knowledge of the decomposition of $\varphi_{\ast} \mathcal{O}_X$,  as provided by \textnormal{Theorem~A} in this paper.

The next statement generalizes \cite[The Main Theorem]{CIK24b} to the case of $m:1$ coverings $\varphi : X \to Y$, where $X$ lies on a cone and passes through its vertex, being algebraically equivalent to the intersection of the cone with a hypersurface of degree $m \geq 3$ and a line from its ruling.

\begin{thmBm}
Assume that $e$, $m$, and $\gamma$ are integers satisfying $e \geq 4 \gamma +5$, $m \geq 3$, and $\gamma \geq 3$. Define
\[
  d:= me + 1 , \qquad  g := \binom{m}{2} e + m \gamma , \quad \mbox{ and } \quad R:= e - \gamma + 1 .
\]
Then the Hilbert scheme $\mathcal{I}_{d, g, R}$ contains a \emph{non-reduced}, irreducible component $\mathcal{D}^{\prime}_{m,R}$ such that:
\begin{enumerate}[label=(\roman*), leftmargin=*, font=\rmfamily]
  \item $\dim \mathcal{D}^{\prime}_{m,R} = g + (m-1)(e - 2(\gamma-1)) - 1 + 3\gamma-2 + R^2 + 2R$;

  \item at a general point $[X] \in \mathcal{D}^{\prime}_{m,R}$, the dimension of the Zariski tangent space satisfies
 \[
  \dim T_{[X]} \mathcal{D}^{\prime}_{m,R} = \dim \mathcal{D}^{\prime}_{m,R} + 1 \, ;
 \]

  \item a general point $[X] \in \mathcal{D}^{\prime}_{m,R}$ corresponds to a curve $X \subset \mathbb{P}^R$ lying on a cone $F$ over a curve $Y \subset \mathbb{P}^{R-1}$ of degree $e$ and genus $\gamma$, satisfying the following properties:
  \begin{enumerate}[label=\theenumi.\arabic*.]
   \item $X \subset \mathbb{P}^{R}$ is linearly normal and passes through the vertex $P$ of the cone $F$;

   \item there exists a line $\ell$ from the ruling of $F$ that is tangent to $X$ at $P$, with intersection multiplicity two;

   \item the projection from $P$ onto the hyperplane containing $Y$ induces a morphism $\varphi : X \to Y$ of degree $m$ ;

   \item the ramification divisor $R_{\varphi}$ is linearly equivalent to the divisor cut on $X$ by a hypersurface of degree $m-1$, together with the points $Q_1 + Q_2 + \cdots + Q_{m-1}$, where $Q_1, \ldots , Q_{m-1}$ are the remaining $m-1$ points (besides $P$) where the tangent line $\ell$ intersects $X$.
  \end{enumerate}
\end{enumerate}
\end{thmBm}

\medskip

The proof of \textnormal{Theorem~B$_m$} is similar to the proof of \cite[The Main Theorem]{CIK24b}. We present it in \textnormal{Section~\ref{Sec4}}.

\subsection*{Conventions and notation}

Throughout the paper, we work over the field $\mathbb{C}$. By a \emph{curve}, we mean a smooth, integral projective curve. Given a line bundle $L$ on a smooth projective variety $X$, or a divisor $\Delta$ associated with $L$, we denote the complete linear series by $|L| := \mathbb P\left(H^0(X,L)\right)$. For a line bundle $L$ and a divisor $\Delta$ on $X$, we sometimes abbreviate the tensor product $L \otimes \mathcal{O}_X (\Delta)$ by simply writing $L(\Delta)$. The symbol $\sim$ denotes linear equivalence of divisors.
For a vector bundle $\mathcal{V}$ on a variety $Y$, we follow Hartshorne's convention for the projectivization $\mathbb{P} (\mathcal{V})$. If $f : \mathbb{P} (\mathcal{V}) \to Y$ is the natural projection and $D$ is a divisor on $Y$, we denote its pullback by $f^{\ast} D$. When $X$ varies in a family, we denote by  $[X]$ the corresponding point in the Hilbert scheme parametrizing the family. For any further definitions and properties not explicitly introduced in the paper, we refer the reader to \cite{Hart77}.

\medskip

\section{Decomposition of the Tschirnhausen module for curves on decomposable ruled surfaces}\label{Sec2}

\subsection{Preliminaries on Tschirnhausen modules and decomposable ruled surfaces}\label{Sec2.1}

The theory of $m:1$ covers in algebraic geometry has been developed by Miranda in \cite{Mir85} in the case $m = 3$, and by Casnati and Ekedahl in general, see \cite{CE96}. Partial results, examples and applications have been obtain by many more. Here we list a few basic facts about coverings of curves following \cite{CE96} and \cite{Hart77}.

Let $\varphi : X \to Y$ be an $m:1$ cover with $m \geq 2$, where $X$ and $Y$ are smooth integral curves of genus $g$ and $\gamma$, respectively. This means that $\varphi$ is finite, flat, and surjective. The covering induces a short exact sequence of vector bundles on $Y$
\begin{equation*}\label{Sec2_SES_Tschirn_Mod}
 0 \to \mathcal{O}_Y \xrightarrow{\varphi^{\sharp}} \varphi_{\ast} \mathcal{O}_X \to \mathcal{E}^{\vee} \to 0 \, ,
\end{equation*}
where $\mathcal{E}^{\vee}$ is the so-called \emph{Tschirnhausen module}. Locally, any section  $x \in \varphi_{\ast} \mathcal{O}_X$ defines an $\mathcal{O}_Y$-linear endomorphism of $\varphi_{\ast} \mathcal{O}_X$ via multiplication. Since $\varphi_{\ast} \mathcal{O}_X$ is a locally free $\mathcal{O}_Y$ module of rank $m$, this multiplication is represented by an $m \times m$ matrix $M_x$. Define the trace map $\tau_{m} : \varphi_{\ast} \mathcal{O}_X \to \mathcal{O}_Y$ locally by $\tau_{m}  (x) = \frac{1}{m}trace (M_x)$. This map extends globally and provides a splitting of the sequence {\rm (\ref{Sec2_SES_Tschirn_Mod})}, as $\mathcal{E}^{\vee}$ consists of trace-zero elements in $\varphi_{\ast} \mathcal{O}_X$. Therefore, $\mathcal{E}^{\vee}$ is a vector bundle of rank $(m-1)$ on $Y$. For further details, see \cite{HM99}.

Furthermore, if $X$ and $Y$ are smooth curves, it follows from \cite[Ex. IV.2.6, p. 306]{Hart77} that
\[
 (\det \varphi_{\ast} \mathcal{O}_X )^2 \cong \mathcal{O}_Y (-B) ,
\]
where $B$ denotes the branch divisor of the covering. By the Riemann-Hurwitz formula, we have $\deg B = 2(g - 1) - 2m (\gamma - 1)$, which implies that
\[
 \deg \det (\mathcal{E}^{\vee}) = \deg \det (\varphi_{\ast} \mathcal{O}_X) = g - 1 - m (\gamma - 1) \, .
\]
In particular, when $m = 2$, the Tschirnhausen bundle $\mathcal{E}^{\vee}$ is a line bundle  of degree $e := g - 2\gamma + 1$, or equivalently, $\mathcal{E} \cong \mathcal{O}_Y (-E)$ for some divisor $E$ of degree $e$ on $Y$. In our previous papers \cite{CIK21} and \cite{CIK24a}, we constructed degree-two coverings in which $E$ is a non-special, very ample divisor on $Y$; for instance, $e \geq 2\gamma +1$ is sufficient for very ampleness. These coverings were realized as curves on the ruled surface $\mathbb{P} (\mathcal{O}_Y \oplus \mathcal{O}_Y (E))$.

In this section, we further explore the construction of degree-$m$ coverings using decomposable ruled surfaces. We begin by recalling a few basic facts about them, as presented in \cite[Ch.V.2]{Hart77} and \cite{GP2005}.

Let $Y$ be a smooth, irreducible projective curve of genus $\gamma$, and $E$ be a divisor on $Y$ of degree $\deg E = e > 2\gamma$. Consider the decomposable ruled surface
\begin{equation*}\label{Sec2_f-S-Y}
 f : S = \mathbb{P} (\mathcal{O}_Y  \oplus \mathcal{O}_Y (E)) \to Y \, .
\end{equation*}
with natural projection $f : S \to Y$. The surface $S$ has a section $\sigma_0$ corresponding to the exact sequence
\[
 \sigma_0 \, : \qquad 0 \to \mathcal{O}_Y \to \mathcal{O}_Y(-E) \oplus \mathcal{O}_Y \to \mathcal{O}_Y(-E) \to 0 \, .
\]
Let $Y_0 := \sigma_0 (Y)$. Then $Y_0$ is the \emph{section of minimum self-intersection}, as $Y^2_0 = -e$.
The surface also contains another section $\sigma_1$, associated to the exact sequence
\[
 \sigma_1 \, : \qquad 0 \to \mathcal{O}_Y \to \mathcal{O}_Y \oplus \mathcal{O}_Y(E) \to \mathcal{O}_Y(E) \to 0 \, ,
\]
which corresponds to the surjection
$
 \mathcal{O}_Y(-E) \oplus \mathcal{O}_Y \to \mathcal{O}_Y \to 0
$.
Let $Y_1 := \sigma_1 (Y)$. Then $Y_1 \sim Y_0 + f^{\ast} E$, and the intersection numbers are
\[
 Y_0 \cdot Y_1 = 0 \, , \quad \mbox{ and } \quad
 Y^2_1 = e \, .
\]
For the direct image sheaf of $\mathcal{O}_S (Y_0)$ under $f$, we have
\begin{equation*}\label{Sec2_f*OS(Y0)}
 f_{\ast} \mathcal{O}_S (Y_0) \cong \mathcal{O}_Y (-E) \oplus \mathcal{O}_Y .
\end{equation*}
The Picard group of $S$ satisfies:
\[
 \Pic (S) \cong \mathbb{Z} [H] \oplus f^{\ast} (\Pic (Y)) ,
\]
where $H$ denotes the tautological divisor on $S$, i.e., $\mathcal{O}_S (1) \cong \mathcal{O}_S (H)$.
For any divisor $Z$ on $S$ and any divisor $D$ on $Y$, we have:
\begin{equation*}\label{Sec2_f*OS(Z+f*D)}
 f_{\ast} \mathcal{O}_S (Z+f^{\ast}D) \cong f_{\ast} \mathcal{O}_S (Z) \otimes \mathcal{O}_Y (D) .
\end{equation*}
In particular, since
$
  f_{\ast} \mathcal{O}_S (H) \cong \mathcal{O}_Y \oplus \mathcal{O}_Y (E)
$
 and
$
f_{\ast} \mathcal{O}_S (Y_0 + f^{\ast} E) \cong (\mathcal{O}_Y (-E)\oplus \mathcal{O}_Y) \otimes \mathcal{O}_Y (E) \cong \mathcal{O}_Y \oplus \mathcal{O}_Y (E)
$,
it follows that the tautological divisor $H$ satisfies:
\begin{equation*}\label{Sec2_H}
 H \sim Y_0 + f^{\ast} E .
\end{equation*}

Before proceeding to the proofs of \textnormal{Theorem~A} and \textnormal{Theorem~B}, we use the present setting to establish a few additional results that will be used in Sections~\ref{Sec4} and~\ref{Sec5}.
We will make use of the dimension of two linear series on $S$, as described in the following lemma.

\begin{lemma}\label{Sec2_dim_LinSer_on_S}
Let $q \in Y$ be a point, and let $m \geq 1$ be an integer. Then
\begin{enumerate}[label=(\roman*), leftmargin=*, font=\rmfamily]
 \item $\dim |\mathcal{O}_S (mH)| = \binom{m+1}{2}e - m\gamma + m$;

 \item $\dim |\mathcal{O}_S (mH + f^{\ast} q)| = \binom{m+1}{2}e - m\gamma + 2m$.
\end{enumerate}
\end{lemma}
\begin{proof}
The claims follow easily from \cite[Lemma 35]{GP2005} and the Riemann-Roch theorem, since $E$ is a non-special divisor of $e \geq 2\gamma + 1$ on $Y$, and $S = \mathbb{P} (\mathcal{O}_Y \oplus \mathcal{O}_Y (E))$.
\end{proof}

In Sections~\ref{Sec4} and~\ref{Sec5}, we will consider curves on a cone $F$, which is the image of $S$ under the morphism induced by the tautological line bundle $\mathcal{O}_S (1) \cong \mathcal{O}_S (H)$.
\begin{lemma}\label{Sec2_Lemma_morphism_psi}
\begin{enumerate}[label=(\roman*), leftmargin=*, font=\rmfamily]
 \item The linear series $|\mathcal{O}_S (1)|$ is base point free and defines a morphism
\[
 \Psi := \Psi_{|\mathcal{O}_S (1)|} : S \to \mathbb{P}^R ,
\]
where $R = e-\gamma+1$.

 \item The morphism $\Psi$ is an isomorphism away from $Y_0$, which it contracts to a point.

 \item Geometrically, the image $F := \Psi (S)$ is a cone in $\mathbb{P}^R$ over a smooth curve $Y_e \cong Y$ of degree $e$, embedded in $\mathbb{P}^{R-1}$. The lines in the ruling of $F$ are the images of the fibers $f^{\ast} q$ for $q \in Y$.
\end{enumerate}
\end{lemma}
\begin{proof}
The fact that $|\mathcal{O}_S (1)|$ is base point free follows from \cite[Ex.~V.2.11, p.385]{Hart77}. Lemma \ref{Sec2_dim_LinSer_on_S} gives $\dim |\mathcal{O}_S (1)| = e-\gamma+1 = R$, so the series defines a morphism $\Psi : S \to \mathbb{P}^R$, proving \textnormal{(i)}.

That $\Psi$ is an isomorphism away from $Y_0$ follows from \cite[Proposition~23]{GP2005}. Since $H \cdot Y_0 = 0$, we have $\Psi (Y_0) =: p$ is a point, establishing \textnormal{(ii)}.

By Bertini's theorem, a general element $Y_1$ of $|\mathcal{O}_S (Y_0 + f^{\ast} E)| = |\mathcal{O}_S (1)|$ is smooth and integral. Since $Y_1 = \sigma_1 (Y)$, we have $Y_1 \cong Y$. The restriction of $\mathcal{O}_S (1)$ to $Y_1$ defines  a divisor of degree $H \cdot Y_1 = H^2 = e$, so the induced morphism  $\Psi_{|_{Y_1}} : Y_1 \to \mathbb{P}^{R-1}$ is an embedding. Its image $Y_e := \Psi (Y_1) \cong Y$ has degree $e$, and thus $F := \Psi (S)$ is a cone in $\mathbb{P}^R$ with vertex $p := \Psi (Y_0)$, since general hyperplane sections of $F$ correspond to images under $\Psi$ of general elements of $|\mathcal{O}_S (1)|$. Finally, as $H \cdot f^{\ast} q = 1$ for any $q \in Y$, each fiber is mapped by $\Psi$ to a line in the ruling of $F$. This completes the proof.
\end{proof}

\subsection{Proof of Theorem A}\label{Sec2.2}

The conclusion of \textnormal{Theorem~A} follows follows from the proposition below, whose proof we now present.

\begin{prop}\label{ThmA_Prop}
Let $Y$ be a smooth projective curve, and $E$ be a non-special divisor on $Y$. Consider the ruled surface $f : S = \mathbb{P} (\mathcal{O}_Y \oplus \mathcal{O}_Y (E)) \to Y$, and let $X \subset S$ be a smooth curve linearly equivalent to $mH$, where $H$ denotes the tautological divisor on $S$, and $m \geq 2$ is an integer. Let $\varphi := f_{|_X}$ denote the restriction of $f$ to $X$. Then
\begin{enumerate}[label=(\roman*), leftmargin=*, font=\rmfamily]
 \item The morphism $\varphi$ is finite of degree $m$, and its direct image sheaf decomposes as:
 \begin{equation}\label{ThmA_Prop_Decomp}
  \varphi_{\ast} \mathcal{O}_X \cong \mathcal{O}_Y \oplus \mathcal{O}_Y (-E)  \oplus \cdots \oplus \mathcal{O}_Y (-(m-1)E) .
 \end{equation}

 \item For each integer $k$, the first higher direct image sheaf of $\mathcal{O}_S (-k)$ under $f$ satisfies:
 \begin{equation}\label{ThmA_Prop_Decomp_Coro}
 R^1 f_{\ast} \mathcal{O}_S (-k) \cong
 \begin{cases}
  0   & \mbox{ if } k \leq 1 \, ; \\
  \mathcal{O}_Y (-E) \oplus \cdots \oplus \mathcal{O}_Y (-(k-1)E)  & \mbox{ if } k \geq 2 \, .
 \end{cases}
\end{equation}
\end{enumerate}
\end{prop}
\begin{proof}
The strategy of the proof is to establish \textnormal{(\ref{ThmA_Prop_Decomp})} and \textnormal{(\ref{ThmA_Prop_Decomp_Coro})} simultaneously by induction. To that end, let us denote by $X_m$ a general curve in the linear equivalence class of $mH$, and by $\varphi_m$ the restriction of $f$ to $X_m$. Thus, we have a covering morphism
\[
 \varphi_m : X_m \to Y,
\]
which is the $m:1$ morphism $\varphi$ referred to in the proposition. For simplicity, we will omit the subscript $m$ when it is clear from the context which curve the map $f$ is being restricted to.

We begin by noting the following standard vanishing:
\begin{equation*}
 R^1 f_{\ast} \mathcal{O}_S (k) = 0 \quad \mbox{ if } \quad k \geq -1 \, .
\end{equation*}
Indeed, the restriction of $\mathcal{O}_S (k)$ to each fiber of $f$ is a line bundle of degree $k \geq -1$, and since $h^1 (\mathbb{P}^1, \mathcal{O}_{\mathbb{P}^1} (k)) = 0$, the vanishing follows from Grauert's theorem; see \cite[Corollary~III.12.9]{Hart77}.

We now show that the decomposition in part~(i) of the proposition implies the one in part~(ii).
\begin{lemma}\label{Lemma_ThmA_(1)=>(2)}
Let $m \geq 2$ and $X_m \sim mH$. If $\varphi_{\ast} \mathcal{O}_{X_m} \cong \mathcal{O}_Y \oplus \mathcal{O}_Y (-E)  \oplus \cdots \oplus \mathcal{O}_Y (-(m-1)E)$, then
\[
R^1 f_{\ast} \mathcal{O}_S (-m) \cong \mathcal{O}_Y (-E) \oplus \cdots \oplus \mathcal{O}_Y (-(m-1)E) .
\]
\end{lemma}

\begin{proof}[Proof of the lemma]
Consider the short exact sequence:
\[
 0 \to \mathcal{O}_S (-m)  \to \mathcal{O}_S \to \mathcal{O}_{X_m} \to 0 ,
\]
and apply the pushforward functor $f_{\ast}$ to it. Since $f_{\ast} \mathcal{O}_S (-m) = 0$, we obtain the derived exact sequence:
\[
 0 \to \mathcal{O}_Y \xrightarrow{\iota} \varphi_{\ast} \mathcal{O}_{X_m} \to R^1 f_{\ast} \mathcal{O}_S (-m)  \to R^1 f_{\ast} \mathcal{O}_S \to \cdots \, .
\]
But $R^1 f_{\ast} \mathcal{O}_S = 0$, so we have a short exact sequence:
\begin{equation}
\label{Sec2_ThmA_Lemma_SES_splits}
  0 \to \mathcal{O}_Y \xrightarrow{\iota} \varphi_{\ast} \mathcal{O}_{X_m} \to R^1 f_{\ast} \mathcal{O}_S (-m)  \to 0 \, .
\end{equation}
By assumption, $\varphi_{\ast} \mathcal{O}_{X_m} \cong \mathcal{O}_Y \oplus \mathcal{O}_Y (-E)  \oplus \cdots \oplus \mathcal{O}_Y (-(m-1)E)$, so $\iota$ must be the inclusion
\[
 \iota (s) = (cs, 0, \ldots , 0) \mbox{ for } s \in \mathcal{O}_Y, \mbox{ and some } c \in \mathbb{C}^{\ast},
\]
as the only global section comes from $\mathcal{O}_Y$, while each $\mathcal{O}_Y (-kE)$ has negative degree for $k \geq 1$. Hence, the exact sequence \eqref{Sec2_ThmA_Lemma_SES_splits} splits, and we conclude:
\[
 R^1 f_{\ast} \mathcal{O}_S (-m) \cong \varphi_{\ast} \mathcal{O}_X / \mathcal{O}_Y \cong \mathcal{O}_Y (-E) \oplus \cdots \oplus \mathcal{O}_Y ((-m+1)E)
\]
as claimed.
\end{proof}

We now prove \eqref{ThmA_Prop_Decomp} by induction on $m$.

\begin{enumerate}[label=(\arabic*), leftmargin=*, font=\rmfamily]
\item \textbf{Base case: $m = 2$.}
Consider the short exact sequence:
\[
  0 \to \mathcal{O}_S (-1)  \to \mathcal{O}_S (1) \to \mathcal{O}_{X_2} (1) \to 0 ,
\]
and apply $f_{\ast}$. Since $f_{\ast}\mathcal{O}_S (-1) = R^1 f_{\ast} \mathcal{O}_S (-1) = 0$, we get
\[
 \varphi_{\ast} \mathcal{O}_{X_2} (1) \cong f_{\ast} \mathcal{O}_S (1) \cong \mathcal{O}_Y \oplus \mathcal{O}_Y (E) .
\]
As $\mathcal{O}_{X_2} (1) \cong \mathcal{O}_{X_2} (H_{|_{X_2}}) \cong \mathcal{O}_{X_2} ((Y_0 + f^{\ast})_{|_{X_2}}) \cong \mathcal{O}_{X_2} (\varphi^{\ast} E)$, we deduce
\[
 \varphi_{\ast} \mathcal{O}_{X_2} (1) \cong \varphi_{\ast} \mathcal{O}_{X_2} (\varphi^{\ast} E) \cong \varphi_{\ast} \mathcal{O}_{X_2} \otimes \mathcal{O}_Y (E) ,
\]
hence,
\[
 \varphi_{\ast} \mathcal{O}_{X_2} \cong \varphi_{\ast} \mathcal{O}_{X_2} (1) \otimes \mathcal{O}_Y (-E) \cong \left( \mathcal{O}_Y \oplus \mathcal{O}_Y (E) \right) \otimes \mathcal{O}_Y (-E) \cong \mathcal{O}_Y (-E) \oplus \mathcal{O}_Y .
\]

 \item \textbf{Inductive step.} Assume the result holds for $X_m \sim mH$, i.e.,
  \[
   \varphi_{\ast} \mathcal{O}_{X_m} \cong \mathcal{O}_Y \oplus \mathcal{O}_Y (-E)  \oplus \cdots \oplus \mathcal{O}_Y (-(m-1)E) .
  \]
  Hence, by Lemma~\ref{Lemma_ThmA_(1)=>(2)},
\[
  R^1 f_{\ast} \mathcal{O}_S (-m) \cong \mathcal{O}_Y (-E) \oplus \cdots \oplus \mathcal{O}_Y (-(m-1)E)
\]
holds as well.

\item \textbf{Step $m+1$.}
Consider the exact sequence:
\[
 0 \to \mathcal{O}_S (-m)  \to \mathcal{O}_S(1) \to \mathcal{O}_{X_{m+1}} (1) \to 0 \, ,
\]
and apply $f_{\ast}$, obtaining:
\[
 0 \to f_{\ast} \mathcal{O}_S(1) \to \varphi_{\ast} \mathcal{O}_{X_{m+1}} (1) \to R^1 f_{\ast} \mathcal{O}_S (-m) \to R^1 f_{\ast} \mathcal{O}_S (1) \, .
\]
The last term vanishes,  and by the inductive hypothesis:
\[
 R^1 f_{\ast} \mathcal{O}_S (-m) \cong \mathcal{O}_Y (-E) \oplus \cdots \oplus \mathcal{O}_Y (-(m-1)E) \, .
\]
Since $\mathcal{O}_{X_{m+1}} (1) \cong \mathcal{O}_{X_{m+1}} (\varphi^{\ast} E)$, we obtain by the projection formula
\[
 \varphi_{\ast} \mathcal{O}_{X_{m+1}} (1) \cong \varphi_{\ast} \mathcal{O}_{X_{m+1}} \otimes  \mathcal{O}_Y (E).
\]
Hence, the short exact sequence becomes:
\begin{equation}\label{Sec2_ThmA_SES_Xm+1}
 \begin{aligned}
  0 & \to \mathcal{O}_Y \oplus \mathcal{O}_Y (E) \to \varphi_{\ast} \mathcal{O}_{X_{m+1}} \otimes \mathcal{O}_Y (E)
  \to \mathcal{O}_Y (-E) \oplus \cdots \oplus \mathcal{O}_Y (-(m-1)E) \to 0 .
 \end{aligned}
\end{equation}
Since $\deg E = e \geq 2\gamma+1$, for all summands $\mathcal{O}_Y (-kE)$, $k = 1, \ldots , m$, we have $H^1 (Y, \mathcal{O}_Y (-kE)) = 0$, so we get:
\[
\begin{aligned}
 \Ext^1 & (R^1 f_{\ast} \mathcal{O}_S (-m), \mathcal{O}_Y \oplus \mathcal{O}_Y (E)) \\
 & \cong
 H^1 (Y, \left( \mathcal{O}_Y (E) \oplus \cdots \oplus \mathcal{O}_Y ((m-1)E) \right) \otimes \left( \mathcal{O}_Y \oplus \mathcal{O}_Y (E)) \right)) \\
 & \cong
 H^1 \left(Y,  \mathcal{O}_Y (E) \oplus ( \oplus^2_1 \mathcal{O}_Y (2E)) \oplus \cdots \oplus (\oplus^2_1 \mathcal{O}_Y ((m-1)E)  ) \oplus \mathcal{O}_Y (mE) \right) \\
 & = 0 \, .
\end{aligned}
\]
Therefore, the exact sequence \eqref{Sec2_ThmA_SES_Xm+1} splits and
\[
 \varphi_{\ast} \mathcal{O}_{X_{m+1}} \otimes \mathcal{O}_Y (E) \cong
 \mathcal{O}_Y (E) \oplus \mathcal{O}_Y \oplus \mathcal{O}_Y (-E) \oplus \cdots \oplus \mathcal{O}_Y (-(m-1)E) \, .
\]
Tensoring with $\mathcal{O}_Y (-E)$ gives:
\[
 \varphi_{\ast} \mathcal{O}_{X_{m+1}} \cong
 \mathcal{O}_Y \oplus \mathcal{O}_Y (-E) \oplus \cdots \oplus \mathcal{O}_Y (-mE),
\]
as required.
\end{enumerate}
This completes the proof of \textnormal{Proposition~\ref{ThmA_Prop}}.

\end{proof}

\begin{remark}
We note that one could alternatively apply the general formula for a rank-two vector bundle $E$ over a curve $Y$, with $f : S = \mathbb{P} (\mathcal{E}) \to Y$, namely:
\[
 R^1 f_*(\mathcal{O}_S(k)) \cong
 \begin{cases}
  0     & \mbox{ if } k \geq -1 \\
  \mathrm{Sym}^{-k - 2}(\mathcal{E}^\vee) \otimes \det(\mathcal{E}^\vee) & \mbox{ if } k \leq -2 ,
 \end{cases}
\]
to give a shorter proof of \textnormal{Theorem~A}. However, we choose to proceed via \textnormal{Proposition~\ref{ThmA_Prop}} for the sake of simplicity and self-containment.
\end{remark}

\subsection{Proof of Theorem~B}\label{Sec2.3}

The assumptions on the curve $Y$, the divisor $E$, and the morphism $f : S \to Y$ are the same as in \textnormal{Theorem~A} and \textnormal{Proposition~\ref{ThmA_Prop}}. However, the curve under consideration is now $X \sim mH + f^{\ast} q$, where $q \in Y$ is a point. The associated morphism is $\varphi : f_{|_X} \to Y$, which differs from the setting of \textnormal{Theorem~A}.

We begin by noting that a general curve $X \in |\mathcal{O}_S (mH + f^{\ast} q)|$ is a smooth and integral, and passes through the point $q_0 := Y_0 \cap f^{\ast} q$. According to \cite[Prop.~36]{GP2005}, this point $q_0$ is the unique fixed point of the linear system $|\mathcal{O}_S (mH + f^{\ast} q)|$ on $S$. Since $X$ intersects each fiber of $f$ transversely in $m$ points and passes through $q_0$, the restriction $\varphi := f_{|_X} X \to Y$ defines an $m:1$ covering $X \to Y$.

To prove \textnormal{(i)} in \textnormal{Theorem~B}, we apply the projection formula \cite[Ex.~III.8.3, p.~253]{Hart77} to obtain:
\[
 R^1 f_{\ast} \mathcal{O}_S (-mH - f^{\ast}q) \cong R^1 f_{\ast} \mathcal{O}_S (-mH) \otimes \mathcal{O}_Y (-q) \, .
\]
Using \eqref{ThmA_Prop_Decomp_Coro}, it follows that
\[
 R^1 f_{\ast} \mathcal{O}_S (-mH - f^{\ast}q) \cong \mathcal{O}_Y (-E-q) \oplus \cdots \oplus \mathcal{O}_Y (-(m-1)E-q) .
\]
Now consider the short exact sequence:
\[
 0 \to \mathcal{O}_S (-mH-f^{\ast}q) \to \mathcal{O}_S \to \mathcal{O}_X \to 0 ,
\]
and apply the functor $f_{\ast}$. Since $f_{\ast} \mathcal{O}_S (-mH-f^{\ast}q) = 0$ and $R^1 f_{\ast} \mathcal{O}_S = 0$, we obtain the exact sequence:
\begin{equation}\label{Sec2_ThmB_SES_Xm_i}
 0 \to \mathcal{O}_Y \to \varphi_{\ast} \mathcal{O}_X \to R^1 f_{\ast} \mathcal{O}_S (-mH-f^{\ast}q) \to 0 \, .
\end{equation}
To determine whether the sequence splits, we compute:
\[
\begin{aligned}
 \Ext^1 (R^1 f_{\ast} \mathcal{O}_S (-mH-f^{\ast}q), \mathcal{O}_Y)
 & \cong \Ext^1 (\mathcal{O}_Y (-E-q) \oplus \cdots \oplus \mathcal{O}_Y (-(m-1)E-q), \mathcal{O}_Y) \\ & \cong H^1 (Y, \mathcal{O}_Y (E+q) \oplus \cdots \oplus \mathcal{O}_Y ((m-1)E+q)) \\
 & = 0 \, ,
\end{aligned}
\]
since $\deg (kE + q) > 2\gamma - 2$ for $k \geq 1$. Thus, the short exact sequence \eqref{Sec2_ThmB_SES_Xm_i} splits, and we conclude:
\[
 \varphi_{\ast} \mathcal{O}_X =  \mathcal{O}_Y \oplus \left(\mathcal{O}_Y (-E) \cdots \oplus \mathcal{O}_Y ((-m+1)E) \right) \otimes \mathcal{O}_Y (-q) .
\]

To prove \textnormal{(ii)} in \textnormal{Theorem~B}, observe that ${Y_0}_{|_X} = q_0$, since $X$ contains $q_0$ and
\[
 Y_0 \cdot X = Y_0 \cdot (mH + f^{\ast} q) = 1 .
\]
Consider the exact sequence:
\[
 0 \to \mathcal{O}_S (-(m-1)H-f^{\ast}(E+q)) \to \mathcal{O}_S (Y_0) \to \mathcal{O}_X (q_0) \to 0 \, ,
\]
and apply the functor $f_{\ast}$. As before, we have:
\[
 f_{\ast}\mathcal{O}_S (-(m-1)H-f^{\ast}(E+q)) = 0 \quad \mbox{ and } \quad R^1 f_{\ast}\mathcal{O}_S (Y_0) = 0 ,
\]
while $f_{\ast}\mathcal{O}_S (Y_0) \cong \mathcal{O}_Y \oplus \mathcal{O}_Y (-E)$. Hence, we obtain the exact sequence:
\[
  0 \to \mathcal{O}_Y \oplus \mathcal{O}_Y (-E) \to \varphi_{\ast} \mathcal{O}_X (q_0) \to R^1 f_{\ast} \mathcal{O}_S (-(m-1)H-f^{\ast}(E+q)) \to 0 \, .
\]
By arguments similar to those used in part \textnormal{(i)}, the sequence splits. Using \eqref{ThmA_Prop_Decomp_Coro}, we get:
\[
 R^1 f_{\ast} \mathcal{O}_S (-(m-1)H-f^{\ast}(E+q)) \cong \left[ \mathcal{O}_Y (-E) \oplus \cdots \oplus \mathcal{O}_Y (-(m-2)E) \right] \otimes \mathcal{O}_Y (-E-q) .
\]
Therefore,
\[
 \varphi_{\ast} \mathcal{O}_X (q_0) \cong \mathcal{O}_Y \oplus \mathcal{O}_Y (-E) \oplus  \mathcal{O}_Y (-2E-q) \oplus \cdots \oplus \mathcal{O}_Y (-(m-1)E-q)
\]
as claimed.

This completes the proof of \textnormal{Theorem~B}.

\medskip

\section{Background on Hilbert schemes, normal bundles, and Gaussian Maps}\label{Sec3}

In this section, we collect several auxiliary results that will be used in the proofs of \textnormal{Theorems~A$_m$} and \textnormal{B$_m$}. These include basic facts about the Hilbert scheme of curves, exact sequences involving normal bundles, a description of the Gaussian map and its restrictions to subspaces, as well as a key corollary summarizing results previously established in \cite{CilMir1990}, \cite{CLM96}, and \cite{CIK21}.

\subsection{Essential facts about the Hilbert scheme of curves}\label{Sec31}

We recall a few basic results concerning the Hilbert scheme of smooth projective curves, which will be used in the subsequent sections.

Let $\mathcal{I}_{d,g,r}$ denote the union of irreducible components of the Hilbert scheme whose general points correspond to smooth, irreducible, non-degenerate complex curves of degree $d$ and genus $g$ in $\mathbb{P}^r$. A component of $\mathcal{I}_{d,g,r}$ is called \emph{regular} if it is reduced and of the expected dimension
\[
 \lambda_{d,g,r} := (r+1)d - (r-3)(g-1).
\]
Otherwise, it is called \emph{superabundant}. When the Brill–Noether number
\[
 \rho(d,g,r) := g - (r+1)(g - d + r)
\]
is non-negative, i.e., $\rho(d,g,r) \geq 0$, it is known that $\mathcal{I}_{d,g,r}$ contains a unique component dominating $\mathcal{M}_g$, the moduli space of smooth curves of genus $g$. This component is often referred to as the \emph{distinguished component}. Moreover, when $d \geq \max \{g + r, 2g - 1\}$, the Hilbert scheme $\mathcal{I}_{d,g,r}$ is irreducible.

If $\mathcal{I} \subset \mathcal{I}_{d,g,r}$ is an irreducible component, then the dimension of the Zariski tangent space at a general point $[C] \in \mathcal{I}$ is given by
\[
 \dim T_{[C]} \mathcal{I} = h^0(C, N_{C/\mathbb{P}^r}),
\]
where $N_{C/\mathbb{P}^r}$ is the normal bundle of the embedding $C \subset \mathbb{P}^r$. Furthermore, for any such component $\mathcal{I} \subset \mathcal{I}_{d,g,r}$, we have the lower bound
\[
 \dim \mathcal{I} \geq h^0(C, N_{C/\mathbb{P}^r}) - h^1(C, N_{C/\mathbb{P}^r}) = \lambda_{d,g,r}.
\]
A component $\mathcal{I}$ is said to be \emph{generically smooth} if
\[
 \dim \mathcal{I} = \dim T_{[C]} \mathcal{I} = h^0(C, N_{C/\mathbb{P}^r})
\]
for a general $[C] \in \mathcal{I}$.
Otherwise, $\mathcal{I}$ is said to be \emph{non-reduced}.

An accessible and thorough exposition of these results can be found in \cite{Har82}, \cite{CS87}, and \cite{EH24}. For convenience, we summarize several of the above facts, along with \cite[Theorem~13.6, p.~484]{CS87}, in the following proposition.

\begin{prop}\label{Sec3_Hilb_Sch_irred}
Let $d \geq 1$, $g \geq 0$, and $r \geq 2$ be integers.
\begin{enumerate}[label=(\roman*), leftmargin=*, font=\rmfamily]
 \item For all triples $(d, g, r)$ with $\rho(d, g, r) \geq 0$, there exists a unique irreducible component $\mathcal{I}_0 \subset \mathcal{I}_{d,g,r}$ that dominates $\mathcal{M}_g$. In particular, a general point $[C] \in \mathcal{I}_0$ corresponds to a curve with general moduli.

 \item If $r \leq d - g$ and $d \geq 2g - 1$, then $\mathcal{I}_{d,g,r}$ is irreducible and generically smooth of dimension
 \[
 \dim \mathcal{I}_{d,g,r} = \lambda_{d,g,r} = (r+1)d - (r - 3)(g - 1).
 \]
\end{enumerate}
\end{prop}

\medskip

\subsection{Normal bundles of curves covering smooth curves}\label{Sec32}

The following proposition summarizes two statements proven in \cite[Lemma~4]{CIK21} and \cite[Proposition~3]{CIK24a}.
\begin{prop}\label{Prop_NormBundSeq}
Let $X \subset \mathbb{P}^r$, $r \geq 3$, be a smooth irreducible projective curve and $p \in \mathbb{P}^r$. Consider the projection $\pi_p$ with center $p$ to a hyperplane $\mathbb{P}^{r-1} \subset \mathbb{P}^r$
\[
 \pi_p : X \to \mathbb{P}^{r-1} \, .
\]
Assume that the image of $X$ is a smooth irreducible curve $Y \subset \mathbb{P}^{r-1}$. Denote by $N_{X / \mathbb{P}^r}$ the normal bundle of $X$ in $\mathbb{P}^r$, and by $N_{Y / \mathbb{P}^{r-1}}$ the normal bundle of $Y$ in $\mathbb{P}^{r-1}$. Let $\varphi : X \to Y$ be the morphism induced by $\pi_p$.
\begin{enumerate}[label=(\alph*), leftmargin=*, font=\rmfamily]
 \item If $p \notin X$, then
\begin{equation}\label{NormBundSeq_p_away}
 0 \to \mathcal{O}_X (1) \otimes \mathcal{O}_X (R_{\varphi}) \to N_{X / \mathbb{P}^r} \to \varphi^{\ast} N_{Y / \mathbb{P}^{r-1}} \to 0 \, ,
\end{equation}
where $R_{\varphi}$ is the ramification divisor of the morphism $\varphi : X \to Y$.

 \item If $p \in X$, then
\begin{equation}\label{NormBundSeq_p_on}
 0 \to \mathcal{O}_X (1) \otimes \mathcal{O}_X (R_{\varphi} + 2p) \to N_{X / \mathbb{P}^r} \to \varphi^{\ast} N_{Y / \mathbb{P}^{r-1}}  \otimes \mathcal{O}_X (p) \to 0 \, ,
\end{equation}
where $R_{\varphi}$ is the ramification divisor of the morphism $\varphi : X \to Y$.
\end{enumerate}
\end{prop}
\begin{remark}
When $p \notin X$ and $\varphi : X \to Y$ is an isomorphism, then $R_{\varphi} = 0$ and the exact sequence \eqref{NormBundSeq_p_away} reduces to the well-known exact sequence
\begin{equation*}\label{NormBundSeq_p_away_iso}
 0 \to \mathcal{O}_X (1) \to N_{X / \mathbb{P}^r} \to \varphi^{\ast} N_{Y / \mathbb{P}^{r-1}} \to 0 \, .
\end{equation*}
When $p \in X$ and $\varphi : X \to Y$ is an isomorphism, then again $R_{\varphi} = 0$ and the exact sequence \eqref{NormBundSeq_p_on} reduces to another known exact sequence
\begin{equation*}\label{NormBundSeq_p_on_iso}
 0 \to \mathcal{O}_X (1) \otimes \mathcal{O}_X (2p) \to N_{X / \mathbb{P}^r} \to \varphi^{\ast} N_{Y / \mathbb{P}^{r-1}}  \otimes \mathcal{O}_X (p) \to 0 \, ,
\end{equation*}
see for example \cite[Proposition~6]{FS19}.
\end{remark}

\medskip

\subsection{The Gaussian map and its restrictions in the context of projections}\label{Sec33}

Suppose $Y$ is a smooth curve and $M$, $L$ are line bundles on $Y$. Consider the multiplication map
\[
 \mu_{M,L} : H^0(Y, M) \otimes H^0(Y, L) \to H^0(Y, M \otimes L).
\]
Wahl introduced in \cite{Wahl90} the \emph{Gaussian map} $\Phi_{M,L}$, a version of which we will use in the proof of \textnormal{Theorem~A$_m$}. We briefly recall its local description.

Let $\mathscr{R}(M,L)$ denote the kernel of $\mu_{M,L}$. Take $\alpha = \sum_i m_i \otimes \ell_i \in \mathscr{R}(M,L)$ and let $U$ be an affine open subset of $Y$. Suppose $M|_U = \mathcal{O}_U \cdot S$ and $L|_U = \mathcal{O}_U \cdot T$, so that locally $m_i = a_i S$ and $\ell_i = b_i T$ for $a_i, b_i \in \Gamma(U, \mathcal{O}_Y)$. Since $\alpha \in \mathscr{R}(M,L)$, we have $\sum_i a_i b_i = 0$. The Gaussian map $\Phi_{M,L}$ is then defined locally by
\[
 \Phi_{M,L}(\alpha)|_U := \sum_i (a_i\, d b_i - b_i\, d a_i)\, S \otimes T \in H^0(U, \omega_Y \otimes M \otimes L).
\]
One can verify that if $\sum_i m_i \otimes \ell_i = 0$, then $\Phi_{M,L}(\sum_i m_i \otimes \ell_i) = 0$, so the local definition extends to a global map.

Now, let $V \subset H^0(Y, L)$ be a vector subspace. Consider the restricted multiplication map
\[
 \mu_{M,V} := {\mu_{M, L}}_{|_{H^0 (Y,M) \otimes V}} : H^0 (Y, M) \otimes V \to H^0 (Y, M \otimes L) ,
\]
and denote its kernel by $\mathscr{R}(M, V)$. The \emph{restricted Gaussian map}
\[
 \Phi_{M,V} : \mathscr{R}(M,V) \to H^0(Y, \omega_Y \otimes M \otimes L)
\]
is defined as the restriction of $\Phi_{M,L}$ to $\mathscr{R}(M, V) \subset \mathscr{R}(M, L)$.

\medskip

The following proposition extends results previously obtained by Wahl, Ciliberto, Lopez, and Miranda to the setting of restricted Gaussian maps.

\begin{prop}\label{Sec3_PropNormGaussMap}
Let $Y$ be a smooth curve of genus $\gamma \geq 1$, and let $L$ be a very ample line bundle on $Y$. Suppose $V \subseteq H^0(Y, L)$ is a linear subspace of dimension $\dim V \geq 4$, such that the associated morphism $\Psi_V : Y \to \mathbb{P}(V^\vee)$ is an embedding. Assume that for some integer $k \geq 1$, the multiplication map
\[
  \mu_{\omega_Y \otimes L^{k-1}, V} : H^0(Y, \omega_Y \otimes L^{k-1}) \otimes V \to H^0(Y, \omega_Y \otimes L^k)
\]
is surjective. Let $N_Y$ denote the normal bundle of $Y \subset \mathbb{P}(V^\vee)$. Then:
\begin{enumerate}[label=(\alph*), leftmargin=*, font=\rmfamily]
 \item If $k=1$, then $h^0(N_Y(-1)) = \dim V + \corank \Phi_{\omega_Y, V}$ ;

 \item If $k \geq 2$, then $h^0(N_Y(-k)) = \corank \Phi_{\omega_Y \otimes L^{k-1}, V}$ ;

 \item If $\corank \Phi_{\omega_Y, V} = 0$, then $h^0(N_Y(-k)) = 0$ for all $k \geq 2$.
\end{enumerate}
\end{prop}

\begin{proof}
Statements \textnormal{(a)} and \textnormal{(b)} extend the results of \cite[Proposition (1.2), pp.~420--421]{CilMir1990} to the setting of restricted Gaussian maps and can be proven by very similar arguments. Similarly, statement \textnormal{(c)} generalizes \cite[Proposition (2.15), pp.~170–171]{CLM96}. For reader's convenience, we include a proof of \textnormal{(c)} here.

Consider the  following commutative diagram:
\[
\begin{CD}
\mathscr{R}(\omega_Y, V) \otimes V @>{\Phi_{\omega_Y, V} \otimes \mathrm{id}_V}>> H^0(\omega_Y^2 \otimes L) \otimes V \\
@V{\nu}VV @V{\mu}VV \\
\mathscr{R}(\omega_Y \otimes L, V) @>{\Phi_{\omega_Y \otimes L, V}}>> H^0(\omega_Y^2 \otimes L^2)
\end{CD}
\]
Here, $\nu$ is the natural inclusion, and $\mu$ is the multiplication map. Since $\dim V \geq 4$, it follows from \cite[Theorem~(4.e.1), p.~162]{Green84} that $\mu$ is surjective. Given that $\corank \Phi_{\omega_Y,V} = 0$, the composition $\mu \circ (\Phi_{\omega_Y, V} \otimes \mathrm{id}_V)$ is also surjective. By the commutativity of the diagram, it follows that $\Phi_{\omega_Y \otimes L, V}$ is surjective, and hence $H^0(N_Y(-2)) = 0$.
To deduce the vanishing $H^0(N_Y(-k)) = 0$ for all $k \geq 2$, consider the short exact sequences:
\[
 0 \to N_Y(-k-1) \to N_Y(-k) \to \mathcal{O}_{Y \cap H} \to 0,
\]
where $H \subset \mathbb{P}(V^\vee)$ is a general hyperplane. Taking global sections and applying induction on yields the desired vanishing.
\end{proof}

We have the following important corollary of \textnormal{Proposition~\ref{Sec3_PropNormGaussMap}}.

\begin{coro}
\label{Sec3_PropNormGaussMap_Coro}
Let $Y$ be a general smooth curve of genus $\gamma \geq 10$, and let $E$ be a general divisor of degree $e \geq 2\gamma + 1$ on $Y$. Let $V \subseteq H^0(Y, \mathcal{O}_Y (E))$ be a general linear subspace of dimension $r := \dim V \geq 4$, and consider the embedding
\[
  \Psi_V : Y \to \mathbb{P} (V^{\vee}) =: \mathbb{P}^{r-1}
\]
induced by $V$. Let $N_{Y}$ denote the normal bundle of $Y \subset \mathbb{P}^{r-1}$.
\begin{enumerate}[label=(\alph*), leftmargin=*, font=\rmfamily]
 \item If $V = H^0(Y, \mathcal{O}_Y (E))$, then:
  \begin{enumerate}[label=(\roman*), leftmargin=*, font=\rmfamily]
    \item the Gaussian map $\Phi_{\omega_Y, V}$ is surjective ;

    \item $h^0(N_Y(-1)) = r$ ;

    \item $h^0(N_{Y}(-k)) = 0$ for all $k \geq 2$.
  \end{enumerate}
  \item If $V \subsetneq H^0(Y, \mathcal{O}_Y (E))$, then statements \textnormal{(i)}-\textnormal{(iii)} remain valid provided that
\[
  r \geq \frac{2(e+2(\gamma-1))}{\gamma} .
\]
\end{enumerate}
\end{coro}
\begin{proof}
In case \textnormal{(a)}, the map $\Phi_{\omega_Y, V}$ coincides with the standard Gaussian map $\Phi_{\omega_Y, \mathcal{O}_Y (E)}$, whose surjectivity follows from \cite[Proposition~(2.8)]{CLM96}. Statements \textnormal{(ii)} and \textnormal{(iii)} then follow directly from Proposition~\ref{Sec3_PropNormGaussMap}.

In case \textnormal{(b)}, assume $r \geq \frac{2(e + 2(\gamma-1))}{\gamma}$. The surjectivity of the restricted Gaussian map $\Phi_{\omega_Y, V}$ is established using the method in \cite[Proposition~7]{CIK21}, which we outline below.

Consider the multiplication map
\[
 \mu : H^0(Y, \omega_Y) \otimes H^0(Y, \mathcal{O}_Y(E)) \to H^0(Y, \omega_Y(E)).
\]
Let $\mu_V$ denote the restriction of $\mu$ to $H^0(Y, \omega_Y) \otimes V$.
We will use that the restricted Gaussian map
\[
 \Phi_{\omega_Y, V} : \mathscr{R}(M,V) \to H^0(Y, \omega_Y^2(E))
\]
is of maximal rank when $V$ is general, as it is shown in \cite{Bal1995}. By \cite[Proposition~4.3]{Bal1995}, we have
\[
 \dim \mathscr{R}(M,V) = \max\left\{0, \dim \mathscr{R}(M,\mathcal{O}_Y (E)) - h^0(Y, \omega_Y)\big(h^0(Y, \mathcal{O}_Y(E)) - \dim V\big)\right\}.
\]
Since $E$ is general of degree $e \geq 2\gamma + 1$, Green's result \cite{Green84} ensures that $\mu$ is surjective. Thus,
\[
 \dim \mathscr{R}(M,\mathcal{O}_Y (E)) = h^0(Y, \omega_Y)\cdot h^0(Y, \mathcal{O}_Y(E)) - h^0(Y, \omega_Y(E)) = (e - \gamma - 1)(\gamma - 1).
\]
Since $r \geq \frac{2(e + 2(\gamma-1))}{\gamma} > \frac{e + \gamma - 1}{\gamma}$, it follows that
\[
 \dim \mathscr{R}(M,\mathcal{O}_Y (E)) - h^0 (Y, \omega_Y) (h^0 (\mathcal{O}_Y (E)) - \dim V) = -e - \gamma + 1 + r \gamma > 0 .
\]
Therefore,
\[
 \dim \mathscr{R}(M,V) = -e - \gamma + 1 + r \gamma .
\]
If $\Phi_{\omega_Y, V}$ were not surjective, it would have to be injective and its image to be properly contained in $H^0 (Y, \omega^2_Y (E))$, which would imply
\[
 -e - \gamma + 1 + r \gamma = \dim \mathscr{R}(M,V) < \dim H^0 (Y, \omega^2_Y (E)) = e + 3(\gamma - 1) .
\]
The last however is impossible due to $r \geq \frac{2(e + 2(\gamma - 1))}{\gamma}$. Therefore, $\Phi_{\omega_Y, V}$ must be surjective. The remaining claims follow again from Proposition~\ref{Sec3_PropNormGaussMap}.
\end{proof}

\medskip

\section{Generically smooth components - proof of Theorem A\texorpdfstring{$_m$}{m}}\label{Sec4}

In this section, we provide the proof of \textnormal{Theorem~A$_m$}. For the reader's convenience, we begin by recalling the main objects that will be referenced frequently throughout this section:

\begin{itemize}[label=\textbullet, leftmargin=1cm, font=\rmfamily]
\item $Y$ is a smooth irreducible projective curve of genus $\gamma \geq 10$;

\item $E$ is a divisor on $Y$ of degree $e = \deg E \geq 2\gamma + 1$;

\item $S$ is the ruled surface $S = \mathbb{P}(\mathcal{O}_Y \oplus \mathcal{O}_Y(E))$, with natural projection map $f : S \to Y$;

\item $Y_0$ is the section of $S$ of minimum self-intersection, satisfying $Y_0^2 = -e$;

\item $H \sim Y_0 + f^{\ast}E$ is the tautological divisor, satisfying $\mathcal{O}_S (H) \cong \mathcal{O}_S (1)$ and $f_{\ast} \mathcal{O}_S (H) \cong \mathcal{O}_Y \oplus \mathcal{O}_Y (E)$;

\item $d = me$, $g = \binom{m}{2}e + m\gamma + 1 - m$, and $R = e - \gamma + 1$.
\end{itemize}

Our approach to proving \textnormal{Theorem~A$_m$} is similar to the strategy used in the proof of \cite[Theorem A]{CIK21}, though we introduce several technical - yet we believe meaningful - improvements. We begin by constructing a family of curves $\mathcal{F}_{m, R}$ in $\mathbb{P}^R$ that lie on cones, as described in the statement of the theorem. We then define the family $\mathcal{F}_{m,r}$ consisting of curves in $\mathbb{P}^r$, obtained as general projections of the curves in $\mathcal{F}_{m, R}$.
We will show that the closures in the Hilbert scheme of the loci parametrizing these families are precisely the components described in \textnormal{Theorem~A$_m$}.

\subsection{Construction of the families \texorpdfstring{$\mathcal{F}_{m, R}$}{FmR} and \texorpdfstring{$\mathcal{F}_{m, r}$}{Fmr}}

We define $\mathcal{F}_{m, R}$ as the family of curves that arise as the images under the morphism $\Psi$ determined by the linear system $|\mathcal{O}_S (1)|$ of the smooth irreducible curves in the linear system $|\mathcal{O}_S (m)|$ on $S$. By \textnormal{Lemma~\ref{Sec2_Lemma_morphism_psi}}, the image $F_R := \Psi (S)$ is a cone in $\mathbb{P}^R$ over a smooth curve of genus $\gamma$ and degree $e$ embedded in $\mathbb{P}^{R-1}$. In fact, a general hyperplane section of $F_R$ corresponds to the image of a general element of $|\mathcal{O}_S (1)|$ and thus gives an embedding of $Y$ via a general line bundle of degree $e$. Each curve $[X] \in \mathcal{F}_{m, R}$ is the image of curve $X_m \in mH$, where the morphism
\[
 \Psi : S \to \mathbb{P} (H^0 (S, \mathcal{O}_S (mH))^{\vee}) =: \mathbb{P}^{R}
\]
realizes $X$ as a curve on the cone $F_R \subset \mathbb{P}^R$. We can therefore view $X \subset F_R \subset \mathbb{P}^R$ as being cut on $S$ by a hypersurface of degree $m$.

We make the following straightforward observation explicit.

\begin{claim}
The curve $X \subset \mathbb{P}^{R}$ is linearly normal.
\end{claim}
\begin{proof}[Proof of the claim]
The statement follows from the exact sequence
\[
 0 \to \mathcal{O}_S (1-m) \to \mathcal{O}_S (1) \to \mathcal{O}_{X_m} (1) \to 0
\]
together with the vanishing of $H^0(S, \mathcal{O}_S(1 - m))$ and $H^0(S, \mathcal{O}_S(1 - m))$ when $m \geq 2$. Since $m \geq 2$, $H^0(S, \mathcal{O}_S(1 - m)) = 0$. To verify that $H^1(S, \mathcal{O}_S(1 - m)) = 0$, we use
\[
 H^1 (S, \mathcal{O}_S (1-m)) = H^0 (Y, R^1 f_{\ast} \mathcal{O}_S (1-m)) = 0 ,
\]
because, by \eqref{ThmA_Prop_Decomp_Coro},
\[
 R^1 f_{\ast} \mathcal{O}_S (1-m) \cong \mathcal{O}_Y (-E) \oplus \cdots \oplus \mathcal{O}_Y (-(m-1)E) ,
\]
and none of summand has global sections when $e \geq 2\gamma - 1$. Finally, \textnormal{Lemma~\ref{Sec2_dim_LinSer_on_S}} implies $h^0 (X_m, \mathcal{O}_{X_m} (1)) = h^0 (S, \mathcal{O}_S (1)) = R+1$.
\end{proof}

To compute the dimension of the family $\mathcal{F}_{m,R}$, we consider the following contributions:

\noindent $\dim \mathcal{F}_{m,R} = $
\begin{itemize}[font=\sffamily, leftmargin=1.3cm, style=nextline]
     \item[$ + $] $3\gamma - 3$ \ : \ number of moduli of smooth curves $Y \in \mathcal{M}_{ \gamma }$;

     \item[$ + $] $\gamma$ \ : \ number of parameters for line bundles $\mathcal{O}_{Y} (E) \in \Pic (Y)$ of degree $e \geq 2\gamma - 1$, necessary to determine the ruled surface $S = \mathbb{P} (\mathcal{O}_{Y} \oplus \mathcal{O}_{Y} (E))$;

     \item[$ + $] $\binom{m+1}{2}e - m\gamma + m = \dim |\mathcal{O}_S (mH)|$ \ : \ number of parameters for choosing a curve  $X_m \sim mH$;

     \item[$ + $] $(R+1)^2 - 1$ \ : \ the dimension of the automorphism group $\Aut(\mathbb{P}^R)$;

     \item[$ - $] $(e- \gamma + 2)$ \ : \ the dimension of the subgroup $G_F \subset \Aut(\mathbb{P}^R)$ that fixes the scroll $F$, see \cite[Lemma~6.4, p.~148]{CCFM2009}.
\end{itemize}
A straightforward computation yields:
\begin{equation}\label{Sec3_DimFmR}
 \dim \mathcal{F}_{m,R} = g+(m-1)e+(-2m+5)(\gamma - 1)-1 + R^2 + 2R .
\end{equation}

\medskip

Next, we define the family $\mathcal{F}_{m,r}$ for $r < R$. This family consists of curves obtained as projections of curves in $\mathcal{F}_{m,R}$ from a general linear subspace $\Lambda \cong \mathbb{P}^{R - r - 1}$ of $\mathbb{P}^R$ onto a general subspace $\mathbb{P}^r \subset \mathbb{P}^R$, where $\Lambda$ and $\mathbb{P}^r$ are complementary in $\mathbb{P}^R$.

We first observe that under such a projection, the image of the cone $F_R \subset \mathbb{P}^R$ is again a cone, denoted $F_r \subset \mathbb{P}^r$. Indeed, for a point $z \in F_R$, its image $\pi_{\Lambda}(z)$ is the unique point of intersection of $\mathbb{P}^r$ with the linear span $\langle \Lambda, z \rangle \cong \mathbb{P}^{R - r}$. Since $F_R$ is a cone with vertex $P \in \mathbb{P}^R$, its ruling consists of lines through $P$. Each such line $\ell \subset F_R$ spans a space $\langle \Lambda, \ell \rangle \cong \mathbb{P}^{R - r + 1}$, which intersects $\mathbb{P}^r$ in a line $\pi_{\Lambda}(\ell) \cong \mathbb{P}^1$. Moreover, all such lines pass through the image $p := \pi_{\Lambda}(P) \in \mathbb{P}^r$, the projection of the vertex.

We formalize this observation as a lemma:

\begin{lemma}\label{Sec4_FR->Fr}
Let $\pi_{\Lambda} : F_R \to \mathbb{P}^r$ be the projection from a general linear subspace $\Lambda \cong \mathbb{P}^{R - r - 1} \subset \mathbb{P}^R$. Then the image $F_r := \pi_{\Lambda}(F_R) \subset \mathbb{P}^r$ is again a cone with vertex $p := \pi_{\Lambda}(P)$, and its ruling consists of the lines $\pi_{\Lambda}(\ell)$, where $\ell$ is a ruling line of $F_R$.
\end{lemma}

We now compute the dimension of the family $\mathcal{F}_{m,r}$ by relating it to the dimension of $\mathcal{F}_{m,R}$, the family of curves before projection. The contributions are as follows:

\noindent
$\dim \mathcal{F}_{m,r} = $
\begin{itemize}[font=\sffamily, leftmargin=1.3cm, style=nextline]
    \item[$+$] $ \dim \mathcal{F}_{m,R} $ \ : \ dimension of the family before projection;

    \item[$-$] $(R+1)^2 - 1$ \ : \ subtract the dimension of $\Aut(\mathbb{P}^R)$, accounting for reparametrizations of the ambient space;

    \item[$+$] $\dim \mathbb{G}(r, R)$ \ : \ number of parameters to choose a linear subspace $\mathbb{P}^r \subset \mathbb{P}^R$;

    \item[$+$] $(r+1)^2 - 1$ \ : \ add the dimension of $\Aut(\mathbb{P}^r)$, accounting for reparametrizations of the target space.
\end{itemize}
Using the previously computed value of $\dim \mathcal{F}_{m,R}$ from~\eqref{Sec3_DimFmR}, we obtain:
\begin{equation}\label{Sec3_DimFmr}
 \dim \mathcal{F}_{m,r} = g + (m - 1)e + (-2m + 5)(\gamma - 1) - 1 + Rr + R + r .
\end{equation}

Now we define $\mathcal{D}_{m,r}$, for $r \leq R$, as the closure in $\mathcal{I}_{me, g, r}$ of the locus parametrizing the family $\mathcal{F}_{m,r}$. From the definition and \eqref{Sec3_DimFmr}, it follows that
\[
 \dim \mathcal{D}_{m,r} = g + (m - 1)e + (-2m + 5)(\gamma - 1) - 1 + Rr + R + r .
\]

\subsection{Computing the dimension of the tangent spaces}

To show that $\mathcal{D}_{m,r}$ forms a generically smooth irreducible component of $\mathcal{I}_{me, g, r}$, it suffices to verify that for a general point $[X] \in \mathcal{F}_{m,r}$, the normal bundle $N_{X/\mathbb{P}^r}$ satisfies:
\[
 h^0 (X, N_{X/\mathbb{P}^r}) = g + (m - 1)e + (-2m + 5)(\gamma - 1) - 1 + Rr + R + r = \dim \mathcal{D}_{m,r} ,
\]
where $N_{X/\mathbb{P}^r}$ denotes the normal bundle of $X \subset \mathbb{P}^r$.

For clarity, and because of certain technical differences in the computations, we treat separately the following two cases:
\begin{description}
 \item[Case 1] $r = R$ (the complete embedding case), and

 \item[Case 2] $\max \left\lbrace \gamma, \frac{2(e+2(\gamma-1))}{\gamma} \right\rbrace \leq r < R$ (the incomplete embedding case)
\end{description}

\medskip

{\bf Case 1.}
We begin by computing $h^0 (X, N_{X/\mathbb{P}^r})$ in the case $r = R$. To do so, we use the fact that projection from the vertex $P$ of the cone $F_R$ to a general hyperplane $\Delta \subset \mathbb{P}^R$ presents $X$ as an $m:1$ covering
\[
 \varphi : X \to \Delta \cap F_R =: Y_R \, .
\]
The embedding $Y_R \subset \Delta \cong \mathbb{P}^{R-1}$ is induced by the restriction ${\mathcal{O}_S (1)}_{|_{Y_1}}$ of the line bundle $\mathcal{O}_S (1)$ to a general smooth curve $Y_1 \in |\mathcal{O}_S (1)|$. Since $Y_R \cong Y_1 \cong Y$ and ${\mathcal{O}_S (1)}_{|_{Y_1}} \cong \mathcal{O}_Y (E)$, where $Y \in \mathcal{M}_{\gamma}$ and $E \in \Div^{e} (Y)$ are general, it follows that $[Y_R] \in \mathcal{I}_{e, \gamma, R-1}$ is a general point. In particular, we have:
\begin{itemize}
 \item $h^0 (Y_R, N_{Y_R/ \mathbb{P}^{R-1}}) = Re - (R-4)(\gamma - 1)$, since $\mathcal{I}_{e, \gamma, R-1}$ is irreducible, generically smooth, and has the expected dimension $\lambda_{e, \gamma, R-1} = Re - (R-4)(\gamma - 1)$, due to $e \geq 2\gamma + 1$ and $R-1 = e - \gamma$, as explained in \textnormal{Proposition~\ref{Sec3_Hilb_Sch_irred}} ;

 \item $h^0 (Y_R, N_{Y_R/ \mathbb{P}^{R-1}} (-1)) = R$, by \textnormal{Corollary~\ref{Sec3_PropNormGaussMap_Coro}(ii)};

 \item $h^0 (Y_R, N_{Y_R/ \mathbb{P}^{R-1}} (-k)) = 0$ for $k \geq 2$, by \textnormal{Corollary~\ref{Sec3_PropNormGaussMap_Coro}(iii)}.
\end{itemize}
Since $X$ is the image under $\Psi$ of a general curve $X_m \in |\mathcal{O}_S (m)|$, and $\Psi$ is an isomorphism onto its image, it follows from \textnormal{Theorem~A} that
\begin{equation}\label{Sec4_varphi*}
 \varphi_{\ast} \mathcal{O}_X \cong \mathcal{O}_{Y_R} \oplus \mathcal{O}_{Y_R} (-1) \oplus \cdots \oplus \mathcal{O}_{Y_R} (-(m-1)).
\end{equation}
Applying the exact sequence \eqref{NormBundSeq_p_away} from \textnormal{Proposition~\ref{Prop_NormBundSeq}}, we obtain
\begin{equation}\label{NormBundSeq_p_away_Sec4}
 0 \to \mathcal{O}_X (1) \otimes \mathcal{O}_X (R_{\varphi}) \to N_{X / \mathbb{P}^R} \to \varphi^{\ast} N_{Y_R / \mathbb{P}^{R-1}} \to 0 ,
\end{equation}
where $R_{\varphi}$ is the ramification divisor of the covering $\varphi : X \to Y_R$. Since $\deg R_{\varphi} = m(m-1)e$ and $\deg X = H \cdot X_m = me$, it follows that
\[
\begin{aligned}
 \deg (\mathcal{O}_X (1) \otimes \mathcal{O}_X (R_{\varphi})) - (2g-2) & = m^2 e - \left[ 2\binom{m}{2}e + 2m\gamma + 2 - 2m - 2 \right] \\
 & = m (e - 2\gamma + 2).
\end{aligned}
\]
By assumption, $e \geq 2\gamma + 1$, so the degree of the line bundle $\mathcal{O}_X (1) \otimes \mathcal{O}_X (R_{\varphi})$ is strictly greater than $2g-2$. Therefore, the cohomology sequence associated to \eqref{NormBundSeq_p_away_Sec4} is exact on global sections. In particular, we obtain:
\[
 h^0 (X, N_{X / \mathbb{P}^R}) = h^0 (X,\mathcal{O}_X (1) \otimes \mathcal{O}_X (R_{\varphi})) + h^0 (X, \varphi^{\ast} N_{Y_R / \mathbb{P}^{R-1}}).
\]
By the Riemann–Roch theorem, we have
\[
 h^0 (X,\mathcal{O}_X (1) \otimes \mathcal{O}_X (R_{\varphi})) = m^2 e - g + 1 = g + m(e - 2(\gamma - 1)) - 1.
\]
To compute $h^0 (X, \varphi^{\ast} N_{Y_R / \mathbb{P}^{R-1}})$, we use that the projection formula:
\[
 h^0 (X, \varphi^{\ast} N_{Y_R / \mathbb{P}^{R-1}}) = h^0 (Y_R, \varphi_{\ast} \mathcal{O}_X \otimes N_{Y_R / \mathbb{P}^{R-1}}) .
\]
From \eqref{Sec4_varphi*}, it follows that
\[
 \varphi_{\ast} \mathcal{O}_X \otimes N_{Y_R / \mathbb{P}^{R-1}} \cong N_{Y_R / \mathbb{P}^{R-1}} \oplus N_{Y_R / \mathbb{P}^{R-1}} (-1) \oplus \cdots \oplus N_{Y_R / \mathbb{P}^{R-1}} (-(m-1)) \, .
\]
As noted above,
\[
 h^0 (Y_R, N_{Y_R / \mathbb{P}^{R-1}} (-1)) = R \mbox{ and } h^0 (Y_R, N_{Y_R / \mathbb{P}^{R-1}} (-k)) = 0 \mbox{ for all } k \geq 2.
\]
Hence,
\[
 \begin{aligned}
  h^0 (X, \varphi^{\ast} N_{Y_R / \mathbb{P}^{R-1}}) & = h^0 (Y_R, N_{Y_R / \mathbb{P}^{R-1}}) + h^0 (Y_R, N_{Y_R / \mathbb{P}^{R-1}} (-1)) \\
  & = Re - (R-4)(\gamma - 1) + R \\
  & = R^2 + 2R - e + 5(\gamma - 1) .
 \end{aligned}
\]
We conclude that:
\[
\begin{aligned}
 h^0 (X, N_{X / \mathbb{P}^R}) & = g + m(e - 2(\gamma - 1)) - 1 + R^2 + 2R - e + 5(\gamma - 1) \\
 & = g + (m-1)(e - 2(\gamma-1)) - 1 + 3(\gamma-1) + R^2 + 2R .
\end{aligned}
\]
This shows that $h^0 (X, N_{X / \mathbb{P}^R}) = \dim \mathcal{F}_{m,R} = \dim \mathcal{D}_{m,R}$, which implies that $\mathcal{D}_{m,R}$ is generically smooth irreducible component of $\mathcal{I}_{me,g,R}$.

\medskip

{\bf Case 2.}
Now we proceed to compute $h^0 (X, N_{X/\mathbb{P}^r})$ in the case where
\[
 \max \left\lbrace \gamma, \frac{2(e+2(\gamma-1))}{\gamma} \right\rbrace \leq r < R .
\]

According to \textnormal{Lemma~\ref{Sec4_FR->Fr}}, each curve $[X] \in \mathcal{F}_{m,r}$ lies on a cone $F_r \subset \mathbb{P}^r$ that arises as the image of a general projection $\pi_{\Lambda}$ from a cone $F_R \subset \mathbb{P}^R$, where $\Lambda \cong \mathbb{P}^{R-r-1}$ is the center of projection. That is, $X = \pi_{\Lambda} (X_R)$, where $X_R$ is the image of a curve $X_m \sim mH$ on $S$, or equivalently the intersection of $F_R$ with a general hypersurface of degree $m$ in $\mathbb{P}^R$.

To complete the calculation of $h^0  (X, N_{X/\mathbb{P}^r})$, we apply \textnormal{Corollary~\ref{Sec3_PropNormGaussMap_Coro}}. For this, we must ensure that a hyperplane section of $F_r$ arises as the projection of a general hyperplane section $Y_R \subset F_R$.

\begin{lemma}\label{Sec4_YR->Yr-general}
Let $H_R \subset \mathbb{P}^R$ be a general hyperplane containing $\Lambda$. Then the restricted projection
\[
 {\pi_{\Lambda}}_{|_{Y_R}} : Y_R \to H_R \cap \mathbb{P}^r \cong \mathbb{P}^{r-1}
\]
is a general linear projection in the sense that it maps $Y_R \subset H_R \cong \mathbb{P}^{R-1}$ to a general linear subspace $\mathbb{P}^{r-1} \subset H_R$.
\end{lemma}
\begin{proof}
Since $R = e - \gamma + 1 \geq \gamma + 2 \geq 12$, and $10 \leq \gamma \leq r < R$, the general center $\Lambda \cong \mathbb{P}^{R-r-1}$ is disjoint from the cone $F_R$. As $H_R$ is general hypersurface containing $\Lambda$, the intersection $Y_R = H_R \cap F_R$ is a smooth curve, isomorphic to the base curve $Y$.

The space $\mathbb{P}^r \subset \mathbb{P}^R $ is general, and since $H_R$ is a hyperplane in $\mathbb{P}^R$, their intersection $\mathbb{P}^r \cap H_R$ is a general codimension-one linear subspace of $\mathbb{P}^r$, and hence a general $\mathbb{P}^{r-1} \subset H_R \cong \mathbb{P}^{R-1}$. Therefore, the restricted projection
 \[
 \pi_{\Lambda}|_{Y_R} : Y_R \to \mathbb{P}^{r-1}
 \]
is a general linear projection of the curve $Y_R \subset \mathbb{P}^{R-1}$ onto $\mathbb{P}^{r-1}$. Since $Y_R$ and $\Lambda \cap Y_R = \emptyset$, this projection is an embedding - as ensured by the assumption $r \geq 10$. Consequently, the image $Y_r := \pi_{\Lambda}(Y_R)$ lies in a general $\mathbb{P}^{r-1} \subset H_R \cong \mathbb{P}^{R-1}$, as claimed.
\end{proof}

In the given range for $r$, a general point $[X] \in \mathcal{F}_{m,r}$ represents a curve lying on a cone $F_r$. As established in \textnormal{Lemma~\ref{Sec4_YR->Yr-general}}, this cone is defined over a base curve $Y_r$, which arises as the image of a curve $Y_R \subset F_R$ under a general projection. The curve $Y_R$ lies in a hyperplane section $H_R \cong \mathbb{P}^{R-1}$, and $Y_r$ is its image in $\mathbb{P}^{r-1}$. Let $p$ denote the vertex of the cone $F_r$; projection from $p$ then realizes $X$ as an $m:1$ cover $\varphi : X \to Y_r$. As in the previous case, this gives rise to an exact sequence for the normal bundle:
\begin{equation*}\label{NormBundSeq_p_away_Sec4_Pr}
 0 \to \mathcal{O}_X (1) \otimes \mathcal{O}_X (R_{\varphi}) \to N_{X / \mathbb{P}^r} \to \varphi^{\ast} N_{Y_r / \mathbb{P}^{r-1}} \to 0 .
\end{equation*}
Moreover, by arguments analogous to those used previously, we have:
\begin{equation*}\label{Sec4_varphi*Pr}
 \varphi_{\ast} \mathcal{O}_X \cong \mathcal{O}_{Y_r} \oplus \mathcal{O}_{Y_r} (-1) \oplus \cdots \oplus \mathcal{O}_{Y_r} (-(m-1)).
\end{equation*}
Since $r \geq \max \left\lbrace \gamma, \frac{2(e+2(\gamma-1))}{\gamma} \right\rbrace$ and $\gamma \geq 10$, it follows that:
\begin{itemize}
 \item $h^0 (Y_r, N_{Y_r/ \mathbb{P}^{r-1}}) = re - (r-4)(\gamma - 1)$, due to $e \geq 2\gamma + 1$, $r-1 < R-1 = e - \gamma$, and \textnormal{Proposition~\ref{Sec3_Hilb_Sch_irred}};

 \item $h^0 (Y_r, N_{Y_r/ \mathbb{P}^{r-1}} (-1)) = r$, by \textnormal{Corollary~\ref{Sec3_PropNormGaussMap_Coro}(ii)}, as $Y_r$ arises from a general projection of $Y_R$;

 \item $h^0 (Y_r, N_{Y_r/ \mathbb{P}^{r-1}} (-k)) = 0$ for $k \geq 2$, by \textnormal{Corollary~\ref{Sec3_PropNormGaussMap_Coro}(iii)}, as $Y_r$ arises from a general projection of $Y_R$.
\end{itemize}
A similar calculation as before gives for $h^0 (X, N_{X/ \mathbb{P}^r})$
\[
 \begin{aligned}
  h^0 (X, N_{X/ \mathbb{P}^r}) & =
  h^0 (X, \mathcal{O}_X (1) \otimes \mathcal{O}_X (R_{\varphi})) + h^0 (X, \varphi^{\ast} N_{Y_r / \mathbb{P}^{r-1}}) \\
  & = g + m(e - 2(\gamma - 1)) - 1 + re - (r-4)(\gamma - 1) + r \\
  & = g + (m-1)(e - 2(\gamma-1)) - 1 + 3(\gamma-1) + (r+1)(R+1) - 1 .
 \end{aligned}
\]
This shows that the dimension of $\dim \mathcal{D}_{m,r}$ coincides with that of the tangent space to $\mathcal{I}_{me, g, r}$ at a general point $[X] \in \mathcal{D}_{m,r}$. Therefore, $\mathcal{D}_{m,r}$ is a generically smooth, irreducible component of  dimension
\[
 \dim \mathcal{D}_{m,r} = g + (m-1)(e - 2(\gamma-1)) - 1 + 3(\gamma-1) + Rr+R+r .
\]
This completes the proof of {\rm Theorem A$_m$.}

\medskip

\begin{remark}
Let $\pi : \mathcal{I}_{d, g, R} \dashrightarrow \mathcal{M}_g$ denote the projection to the moduli space of smooth curves of genus $g$. Then the proof implies that the image of the component $\mathcal{D}_{m,R} \subset \mathcal{I}_{d, g, R}$ under $\pi$ has dimension
\[
\begin{aligned}
 \dim \pi (\mathcal{D}_{m,R}) & = g + (m-1)(e - 2(\gamma-1)) - 1 + 3(\gamma-1) ,   \\
 & = \frac{m^2}{2}e + (m-2) \left[ \frac{e}{2}-(\gamma-1) \right] + 3(\gamma-1) ,
\end{aligned}
\]
The curves $X \sim mH$ on the ruled surface $S = \mathbb{P} (\mathcal{O}_Y \oplus \mathcal{O}_Y (E))$, and thus the curves parametrized by $\mathcal{D}_{m,R}$, define $m\!:\!1$ coverings $\varphi : X \to Y$, for which $\varphi_{\ast} \mathcal{O}_X$ admits the same decomposition as in the case of cyclic covers of degree $m$. However, for $m \geq 3$, a general curve $X \sim mH$ is not a cyclic cover of $Y$.
Indeed, the locus in $\mathcal{M}_g$ of degree-$m$ cyclic covers of curves of genus $\gamma$ has dimension
\[
me + 3(\gamma - 1).
\]
Since $e \geq 2\gamma + 1$, it follows that for $m \geq 3$,
\[
 \dim \pi(\mathcal{D}_{m,R}) > me + 3(\gamma - 1).
\]
This strict inequality shows that the general curve $X \sim mH$ does not arise as a cyclic cover of $Y$.
\end{remark}

\medskip

\section{Nonreduced components - proof of Theorem B\texorpdfstring{$_m$}{m}}\label{Sec5}
In this section, we prove \textnormal{Theorem~B$_m$}, which generalizes the main result of \cite{CIK24b}. The overall structure of the proof closely follows the approach developed in \cite{CIK24b}, with one key distinction: \textnormal{Theorem~B}, established in Section~\ref{Sec2}, allows us to extend the statement of \cite[Proposition~6]{CIK24b} from the case of triple covers to arbitrary $m:1$ covers with $m \geq 2$. This broader framework enables us to compute $h^0 (X, N_{X/\mathbb{P}^R})$ in the general case. The auxiliary results needed for this argument were largely developed in \cite{CIK24b} and remain applicable here.

In what follows, we outline the main steps of the proof, referring to \cite{CIK24b} whenever the arguments are identical, rather than repeating them in full.

The setup is essentially the same as in the proof of \textnormal{Theorem~A$_m$} in Section~\ref{Sec4}, with only minor differences in regards to the numerical invariants. We summarize the relevant differences below:

\begin{itemize}[label=\textbullet, leftmargin=1cm, font=\rmfamily]
\item $Y$ is a smooth, irreducible, projective curve of genus $\gamma \geq 3$;

\item $E$ is a divisor on $Y$ of degree $e = \deg E \geq 4\gamma + 5$;

\item $m \geq 3$, $d := me+1$, $g := \binom{m}{2}e + m\gamma$, and $R := e - \gamma + 1$.
\end{itemize}

The proof of \textnormal{Theorem~B$_m$} proceeds in the following steps:
\begin{itemize}[font=\sffamily, leftmargin=1.8cm, style=nextline]
 \item[{\rm \bf Step I.}] We construct the family $\mathcal{F}^{\prime}_{m,R}$ of curves satisfying property \textnormal{(iii)} in \textnormal{Theorem~B$_m$}; then we consider the closure $\mathcal{D}^{\prime}_{m,R} \subset \mathcal{I}_{d, g, R}$ of the subset parametrizing the family $\mathcal{F}^{\prime}_{m,R}$ and show that
 \[
  \dim \mathcal{F}^{\prime}_{m,R} = g + (m - 1)(e - 2(\gamma - 1)) + 3\gamma - 2 + R^2 + 2R.
 \]

 \item[{\rm \bf Step II.}] For a general curve $X \in \mathcal{F}^{\prime}_{m,R}$, we show:
 \[
 \begin{aligned}
 \dim T_{[X]} \mathcal{D}^{\prime}_{m,R} &= h^0(X, N_{X / \mathbb{P}^{R}}) \\
 &= g + (m - 1)(e - 2(\gamma - 1)) + 3\gamma - 1 + R^2 + 2R \\
 &= \dim \mathcal{D}^{\prime}_{m,R} + 1.
 \end{aligned}
 \]

 \item[{\rm \bf Step III.}] We show that $\mathcal{D}^{\prime}_{m,R}$ is an irreducible component of $\mathcal{I}_{me+1, g, R}$.
\end{itemize}

{\bf Step I. Construction of $\mathcal{F}^{\prime}_{m,R}$.}

Let $\mathcal{F}^{\prime}_{m,R}$ denote the family of curves $X := \Psi (X^{\prime}_m) \subset \mathbb{P}^{R}$, where $X^{\prime}_m \in |\mathcal{O}_S (mH + f^{\ast}q)|$, $H$ is the tautological divisor on the ruled surface $S = \mathbb{P} (\mathcal{O}_Y \oplus \mathcal{O}_Y (E))$, and $\Psi$ is the morphism associated to the linear series $|\mathcal{O}_S (H)|$. Here, $Y \in \mathcal{M}_{\gamma}$ is a general curve of genus $\gamma$, $E \in \Div^e (Y)$ is a general effective divisor of degree $e \geq 4\gamma + 5$, and $q \in Y$ is a point.

The dimension of the family $\mathcal{F}^{\prime}_{m,R}$ is computed as follows:

\noindent $\dim \mathcal{F}^{\prime}_{m,R} = $
\begin{itemize}[font=\sffamily, leftmargin=1.3cm, style=nextline]
     \item[$ + $] $3\gamma - 3$ \ : \ number of moduli of smooth curves $Y \in \mathcal{M}_{ \gamma }$ ;

     \item[$ + $] $\gamma$ \ : \ number of parameters for choosing line bundles $\mathcal{O}_{Y} (E) \in \Pic (Y)$ of degree $e \geq 4\gamma + 5$, necessary to determine the ruled surface $S = \mathbb{P} (\mathcal{O}_{Y} \oplus \mathcal{O}_{Y} (E))$ ;

     \item[$ + $] $1$  \ : \ number of parameters for choosing the point $q \in Y$ ;

     \item[$ + $] $\binom{m+1}{2}e - m\gamma + m = \dim |\mathcal{O}_S (mH + f^{\ast}q)|$ \ : \ number of parameters for choosing the curve  $X \sim mH + f^{\ast}q$

     \item[$ + $] $(R+1)^2 - 1$ \ : \ the dimension of the automorphism group $\Aut(\mathbb{P}^R)$;

     \item[$ - $] $(e- \gamma + 2)$ \ : \ the dimension of the subgroup $G_F \subset \Aut(\mathbb{P}^R)$ that fixes the scroll $F$, see \cite[Lemma~6.4, p.~148]{CCFM2009}.
\end{itemize}
A straightforward computation gives:
\begin{equation*}\label{Sec4_DimFprmR}
 \dim \mathcal{F}^{\prime}_{m,R} = \left[ g+(m-1)(e-2(\gamma - 1)) + 3\gamma - 2 \right] + R^2 + 2R .
\end{equation*}

We define $\mathcal{D}^{\prime}_{m,R}$ to be the closure in $\mathcal{I}_{me+1, g, R}$ of the locus parametrizing the family $\mathcal{F}^{\prime}_{m,R}$.

Recall that each curve $X$ from the family $\mathcal{F}^{\prime}_{m,R}$ lies on a cone $F \subset \mathbb{P}^{R}$ over a curve $Y_R \subset \mathbb{P}^{R-1} \subset \mathbb{P}^{R}$. The curve $Y_R$ is the image $\Psi (Y_1)$ of general $Y_1 \in \mathcal{O}_S (1)$. If $p$ is the vertex of $F$, then $X$ contains $p$, and the projection with center $p$ gives an $m:1$ covering morphism $\varphi : X \to Y_R$.

We remark explicitly the following fact that will be used in {\rm Step III}.

\begin{lemma}\label{Sec5_Lemma_Y_is_proj_norm}
\begin{enumerate}[label=(\roman*), leftmargin=*, font=\rmfamily]
 \item $X$ is linearly normal.

 \item $Y_R$ is projectively normal.
\end{enumerate}
\end{lemma}
\begin{proof}
Consider
\[
 0 \to \mathcal{O}_S (-(m-1)H - f^{\ast}) \to \mathcal{O}_S (H) \to \mathcal{O}_{X^{\prime}_m} (H) \to 0
\]
and apply $\Gamma (S, \cdot)$. Since $H^0 (\mathcal{O}_S (-(m-1)H - f^{\ast})) = H^1 (\mathcal{O}_S (-(m-1)H - f^{\ast})) = 0$, we conclude that $H^0 (\mathcal{O}_S (H)) = H^0 (\mathcal{O}_{X^{\prime}_m} (H)) = \mathbb{C}^{R+1}$, $R = e-\gamma + 1$.

The linear normality of $Y_R$ follows by the Riemann-Roch theorem. Further, the projective normality of $Y_R$ follows from \cite{GL86} since $\deg Y_R = e \geq 4\gamma + 5$.
\end{proof}

\medskip

{\bf Step II.} Computation of the dimension of $T_{[X]} \mathcal{D}^{\prime}_{m,R}$.

We now compute $h^0 (X, N_{X/\mathbb{P}^R})$. To this end, we exploit the fact that projection from the vertex $p$ of the cone $F$ onto a general hyperplane $\Delta \subset \mathbb{P}^R$ realizes $X$ as an $m:1$ covering
\[
 \varphi : X \to \Delta \cap F =: Y_R .
\]
The curve $Y_R \subset \Delta \cong \mathbb{P}^{R-1}$ is embedded by the restriction ${\mathcal{O}_S (1)}_{|_{Y_1}}$, where $Y_1 \in |\mathcal{O}_S (1)|$ is a general smooth section of the ruled surface $S$. Since $Y_R \cong Y_1 \cong Y$ and ${\mathcal{O}_S (1)}_{|_{Y_1}} \cong \mathcal{O}_Y (E)$, where $Y \in \mathcal{M}_{\gamma}$ and $E \in \Div^{e} (Y)$ are general, it follows that $[Y_R] \in \mathcal{I}_{e, \gamma, R-1}$ is a general point. Following the same reasoning as in the complete embedding case in the proof of \textnormal{Theorem~A$_m$}, we obtain:
\begin{itemize}
 \item $h^0 (Y_R, N_{Y_R/ \mathbb{P}^{R-1}}) = Re - (R-4)(\gamma - 1)$, since $\mathcal{I}_{e, \gamma, R-1}$ is irreducible, generically smooth, and has the expected dimension $\lambda_{e, \gamma, R-1} = Re - (R-4)(\gamma - 1)$ ;

 \item $h^0 (Y_R, N_{Y_R/ \mathbb{P}^{R-1}} (-1)) = R$ ;

 \item $h^0 (Y_R, N_{Y_R/ \mathbb{P}^{R-1}} (-k)) = 0$ for $k \geq 2$ .
\end{itemize}

Since $X$ is the image under $\Psi$ of a general curve $X_m \in |\mathcal{O}_S (mH + f^{\ast}q)|$ under $\Psi$, the exact sequence \eqref{NormBundSeq_p_on} from \textnormal{Proposition~\ref{Prop_NormBundSeq}} implies
\begin{equation}\label{NormBundSeq_p_on_Sec5}
 0 \to \mathcal{O}_X (1) \otimes \mathcal{O}_X (R_{\varphi}+2p) \to N_{X / \mathbb{P}^R} \to \varphi^{\ast} N_{Y_R / \mathbb{P}^{R-1}} \otimes \mathcal{O}_X (p) \to 0 .
\end{equation}

We compute the degree of the line bundle $\mathcal{O}_X (1) \otimes \mathcal{O}_X (R_{\varphi}+2p)$. Since $\deg R_{\varphi} = 2g-2 - 2m(\gamma - 1) = (me+2)(m-1)$ and $\deg \mathcal{O}_X (1) = me+1$, we obtain:
\[
 \begin{aligned}
  \deg (\mathcal{O}_X (1) \otimes \mathcal{O}_X (R_{\varphi}+2p))
  & = 2g-2 - 2m(\gamma - 1) + me + 1 + 2 \\
  & = 2g-2 + m (e - 2(\gamma - 1)) + 3 .
 \end{aligned}
\]
Since $e \geq 4\gamma + 5$ by assumption, the sequence \eqref{NormBundSeq_p_on_Sec5} is exact on global sections, and by Riemann–Roch we get:
\[
 \begin{aligned}
  h^0 (X, \mathcal{O}_X (1) \otimes \mathcal{O}_X (R_{\varphi}+2p))
  & = 2g-2 + m (e - 2(\gamma - 1)) + 2 - g + 2 \\
  & = g + m (e - 2(\gamma - 1)) + 2 .
 \end{aligned}
\]
Hence,
\[
 h^0 (X, N_{X / \mathbb{P}^R}) = h^0 (X, \mathcal{O}_X (1) \otimes \mathcal{O}_X (R_{\varphi}+2p)) +   h^0 (X, \varphi^{\ast} N_{Y_R / \mathbb{P}^{R-1}} \otimes \mathcal{O}_X (p)) .
\]
To compute the second term, we apply the projection formula \cite[Ex.~III.8.2, p.252]{Hart77}. By \textnormal{Theorem~B}, we have:
\[
 \varphi_{\ast} \mathcal{O}_X (p) \cong \mathcal{O}_{Y_R} \oplus \mathcal{O}_{Y_R} (-1) \oplus \left[ \mathcal{O}_{Y_R} (-2) \oplus \cdots \oplus \mathcal{O}_{Y_R} (-(m-1)) \right] \otimes \mathcal{O}_{Y_R} (-q) .
\]
Therefore,
\[
\begin{aligned}
 h^0 (X, \varphi^{\ast} N_{Y_R / \mathbb{P}^{R-1}} \otimes \mathcal{O}_X (p)) & = h^0 (Y_R, N_{Y_R / \mathbb{P}^{R-1}} \otimes \varphi_{\ast} \mathcal{O}_X (p)) \\
 & = h^0 (Y_R, N_{Y_R / \mathbb{P}^{R-1}}) + h^0 (Y_R, N_{Y_R / \mathbb{P}^{R-1}}(-1)) \\
 & + \sum^{m-1}_{k=2} h^0 (Y_R, N_{Y_R / \mathbb{P}^{R-1}}(-k) \otimes \mathcal{O}_{Y_R} (-q)) .
\end{aligned}
\]
Since $h^0 (Y_R, N_{Y_R / \mathbb{P}^{R-1}}(-k)) = 0$ for $k \geq 2$, the sum vanishes, and we are left with:
\[
 h^0 (X, \varphi^{\ast} N_{Y_R / \mathbb{P}^{R-1}} \otimes \mathcal{O}_X (p)) = Re - (R-4)(\gamma - 1) + R = R^2 + R + 4(\gamma-1) .
\]
Putting everything together:
\[
 \begin{aligned}
  h^0 (X, N_{X / \mathbb{P}^R}) & =  g + m (e - 2(\gamma - 1)) + 2 + R^2 + R + 4(\gamma-1) \\
  & = g + m (e - 2(\gamma - 1)) + 2 + R^2+2R - (e-2(\gamma-1)) + 3(\gamma-1) \\
  & = \left[ g + (m-1)(e - 2(\gamma-1)) + 3\gamma-1 \right] + R^2+2R .
 \end{aligned}
\]
In other words,
\[
 \dim T_{[X]} \mathcal{D}^{\prime}_{m,R} = h^0 (X, N_{X / \mathbb{P}^R}) = \dim \mathcal{D}^{\prime}_{m,R} + 1 .
\]

\medskip

{\bf Step III.} Proof of $\mathcal{D}^{\prime}_{m,R}$ being an irreducible component of $\mathcal{I}_{d, g, R}$.

The argument in this steps essentially follows the strategy of {\rm Step III} in our earlier work \cite{CIK24b}. By definition,
\[
 \mathcal{D}^{\prime}_{me+1,g,R} \subset \mathcal{I}_{me+1,g,R}
\]
is the closure of the locus parametrizing smooth, integral curves of degree $me+1$ and genus $g$ that lie on cones in $\mathbb{P}^{R}$ over curves parametrized by $\mathcal{I}_{e, \gamma, R-1}$. A general point $[X] \in \mathcal{D}^{\prime}_{me+1,g,R}$ corresponds to a curve in the linear equivalence class $mH + f^{\ast}q$ on the desingularization $S$ of a cone $F \subset \mathbb{P}^R$ over $Y_R \subset \mathbb{P}^{R-1}$, where $[Y_R] \in \mathcal{I}_{e, \gamma, R-1}$ is general. It is clear from construction that  $\mathcal{D}^{\prime}_{me+1,g,R}$ is irreducible. To conclude that it forms an irreducible component of the Hilbert scheme, it suffices to show that any flat deformation of a general curve in $\mathcal{D}^{\prime}_{me+1,g,R}$ remains a curve on a cone in $\mathbb{P}^R$ over a curve from $\mathcal{I}_{e, \gamma, R-1}$.

\begin{lemma}\label{Sec5DeformLemma}
Let $p_{\mathcal{X}} : \mathcal{X} \to T$ be a flat family of projective curves in $\mathbb{P}^R$, and suppose there exists a closed point $t_0 \in T$ such that:
\begin{enumerate}[label=(\roman*), leftmargin=*, font=\rmfamily]
 \item  $\mathcal{X}_{t_0}$ is a smooth, integral, projectively normal curve of genus $g = \binom{m}{2}e + m\gamma$ and degree $me + 1$;

 \item $\mathcal{X}_{t_0}$ lies on a cone $F \subset \mathbb{P}^R$ over a curve $Y$ corresponding to a general point in $\mathcal{I}_{e, \gamma, R-1}$.
\end{enumerate}
Then there exists a neighborhood $U \subset T$ of $t_0$ such that for all closed points $t \in U$, the fiber $\mathcal{X}_t$ is again a curve lying on a cone over a smooth, integral, projectively normal curve of genus $\gamma$ and degree $e$ in $\mathbb{P}^{R-1}$.
\end{lemma}
\begin{proof}
The lemma follows by the same reasoning as in \cite[Lemma~8]{CIK24b}. We briefly sketch its proof.

The generality of $Y$, along with \cite[Proposition~2]{CG99}, implies that the ideal of $\chi_{t_0}$ admits a presentation
\begin{equation*}
 P_2 \to P_1 \to I(\chi_{t_0}) \to 0 ,
\end{equation*}
where
\begin{itemize}
 \item $P_1 = \bigoplus\limits^{R-1}_1 \mathcal{O}_{\mathbb{P}^{R}}(-m-1) \oplus \bigoplus\limits^{\beta_1}_{j=1} \mathcal{O}_{\mathbb{P}^{r}} (-2)$

 \item $P_2 = \bigoplus\limits^{\binom{R-1}{2} + \beta_1}_1 \mathcal{O}_{\mathbb{P}^{R}}(-m-2) \oplus \bigoplus\limits^{\beta_2}_{j=1} \mathcal{O}_{\mathbb{P}^{R}} (-3)$ ,
\end{itemize}
and where $\beta_1$ and $\beta_2$ are the Betti numbers arising in the minimal free resolution
\[
 \cdots \to \bigoplus\limits^{\beta_2}_{j=1} \mathcal{O}_{\mathbb{P}^{R}} (-3) \to \bigoplus\limits^{\beta_1}_{j=1} \mathcal{O}_{\mathbb{P}^{r}} (-2) \to I(Y) \to 0
\]
of the ideal of $Y \subset \mathbb{P}^{R-1}$. Since $Y$ is projectively normal, the family $p_{\mathcal{X}} : \mathcal{X} \to T$ is very flat; see \cite[Ex.~III.9.5, p.~266]{Hart77}. Moreover, because $m \geq 4$, it follows that the zero locus of the degree-two generators of the ideal $I(\chi_t)$, for any flat deformation $\chi_{t}$ of $\chi_{t_0}$, is a cone over a curve of the same genus and degree as $Y$. For additional details, we refer the reader to \cite{CIK24b}, \cite{CLM96} and \cite{Cil87}.
\end{proof}

The lemma implies that every deformation of $X$ remains a curve (with the same invariants) lying on a cone $\tilde{F} \subset \mathbb{P}^R$ over a curve $\tilde{Y}$, where $[\tilde{Y}] \in \mathcal{I}_{e, \gamma, R-1}$ is general. By \cite[Proposition~5]{CIK24b}, such a curve must necessary be in the linear series
\[
 |\mathcal{O}_{\tilde{S}} (m\tilde{H} + \tilde{f}^{\ast} \tilde{q})| ,
\]
where $\tilde{S} = \mathbb{P} (\mathcal{O}_{\tilde{Y}} \oplus \mathcal{O}_{\tilde{Y}} (1))$ is the desingularization of $\tilde{F}$, with natural projection $\tilde{f} : \tilde{S} \to \tilde{Y}$, $\tilde{H}$ is the tautological divisor on $\tilde{S}$ and $\tilde{q} \in \tilde{Y}$ is a point. From the definition of $\mathcal{F}^{\prime}_{m,R}$, it follows that $\tilde{X}$ also belongs to this family. Therefore, $\mathcal{D}^{\prime}_{m,R}$ is an irreducible  component of the Hilbert scheme $\mathcal{I}_{me+1, g, e-\gamma+1}$.

This completes the proof of \textnormal{Theorem~B$_m$}.

\medskip

\begin{remark}
In the proof of \textnormal{Theorem~B$_m$}, we made essential use of the projective normality of the hyperplane sections of the cone $F$ on which the curves $X$ from the family $\mathcal{F}^{\prime}_{m,R}$ lie. This was crucial for establishing the very flatness of the family $p_{\mathcal{X}} : \mathcal{X} \to T$ in \textnormal{Lemma~\ref{Sec5DeformLemma}}. As a result, we could not apply an argument analogous to that used in the proof of \textnormal{Theorem~A$_m$} to deduce that a family $\mathcal{F}^{\prime}_{m,r}$, obtained via projections of the curves in $\mathcal{F}^{\prime}_{m,R}$ into a lower-dimensional projective space $\mathbb{P}^r \subsetneq \mathbb{P}^R$, forms an irreducible component.

It remains an open and interesting question whether such a projected family defines a non-reduced component of the Hilbert scheme or is properly contained in a generically smooth component. More broadly, it would be worthwhile to identify components of the Hilbert scheme of curves whose general member is not linearly normal.
\end{remark}

\end{document}